\documentclass[11pt,letterpaper,reqno]{amsart}
\renewcommand{\baselinestretch}{1.1} 

\usepackage{epsfig}
\usepackage{amsmath}
\usepackage{amssymb}
\usepackage{amsthm}
\usepackage{indentfirst}
\usepackage{xspace}
\usepackage{multirow}
\usepackage{hyperref}
\usepackage{xcolor}
\definecolor{darkred}{rgb}{0.2,0.25,0.75}
\hypersetup{colorlinks=true,urlcolor=darkred,linkcolor=darkred,citecolor=darkred}
\usepackage{verbatim}
\usepackage[letterpaper,margin=1in]{geometry}
\usepackage{mathpazo}
\usepackage{thmtools}
\usepackage[all]{xy}
\usepackage{longtable}
\usepackage{capt-of}
\usepackage{enumitem}
\usepackage{amsmath,amssymb,graphicx,wasysym,tikz}
 \usepackage[utf8]{inputenc}

\usetikzlibrary{arrows, positioning, shapes, decorations.pathreplacing, decorations.markings, calc, matrix}

\tikzset{->-/.style={decoration={
  markings,
  mark=at position #1 with {\arrow{>}}},postaction={decorate}}}

\usepackage{multicol}

 \usepackage{graphicx}
\usepackage{wrapfig}
\usepackage{lipsum}  % generates filler text

\def\impact#1{\bgroup\narrower%\footnotefont
\baselineskip
\footskip\bigbreak
\hrule\vspace{-0.1 in}\medskip\nobreak\noindent \begin{BI}
\renewcommand{\baselinestretch}{1.1}  {\it #1\/}\par\nobreak\end{BI}}
\def\endimpact{\medskip\nobreak \hrule\bigbreak\egroup}

\def\CL{{\mathcal L}}

\setlist{itemsep = 0.20em, topsep = 0.20em}

\declaretheoremstyle[spaceabove=0.25cm,spacebelow=0.25cm,notefont=\normalfont\bfseries, notebraces={(}{)}]{Theorem}
\declaretheoremstyle[spaceabove=0.25cm,spacebelow=0.25cm,bodyfont=\normalfont,notefont=\normalfont\bfseries, notebraces={(}{)}]{noital}
\declaretheoremstyle[spaceabove=0.25cm,spacebelow=0.25cm,bodyfont=\normalfont\color{darkgreen},notefont=\normalfont\bfseries, notebraces={(}{)}]{green}
\declaretheoremstyle[spaceabove=0.25cm,spacebelow=0.25cm,bodyfont=\normalfont,notefont=\normalfont\bfseries,qed=$\qedsymbol$,notebraces={(}{)}]{proofstyle}

\declaretheorem[name=Definition,style=Definition]{Definition}
\declaretheorem[name=Proposition,style=Definition]{Proposition}

\declaretheorem[name=Theorem,style=Definition]{Theorem}
\declaretheorem[name=Corollary,style=Definition]{Corollary}

 \newtheorem{assumption}[Theorem]{Assumption}
\declaretheorem[name=Remark,style=Theorem]{Remark}
\declaretheorem[name=Lemma,style=Definition]{Lemma}

\newtheorem{BI}{Broader Impact}[]

\numberwithin{equation}{section}

\newcommand{\M}{\ensuremath{\mathcal M}}

\newcommand{\CO}{\ensuremath{\mathcal O}}

\newcommand{\GC}{\ensuremath{G_{\mathbb{C}}}}
\newcommand{\MGC}{\ensuremath{\mathcal{M}_{G_{\mathbb{C}}}}}
\newcommand{\MG}{\ensuremath{\mathcal{M}_{G}}}
\newcommand{\GL}{\ensuremath{{\rm GL}}}

\newcommand{\R}{\ensuremath{\mathbb R}}
\newcommand{\C}{\ensuremath{\mathbb C}}

\newcommand{\Z}{\ensuremath{\mathbb Z}}

\usepackage[framemethod=TikZ]{mdframed}
\usepackage{lipsum}
\mdfdefinestyle{MyFrame}{%
    linecolor=black,
    outerlinewidth=1pt,
    roundcorner=5pt,
    innertopmargin=\baselineskip,
    innerbottommargin=\baselineskip,
    innerrightmargin=10pt,
    innerleftmargin=10pt,
    backgroundcolor=gray!6!white}

\usepackage{blindtext}

\usepackage{tcolorbox}

\begin{document}

\bibliographystyle{utphys}

\setcounter{page}{1}

\title[D.~Baraglia ~$\&$ ~L.P. Schaposnik]
{Cayley and Langlands type correspondences\\ for orthogonal Higgs bundles}

\author[D.~Baraglia]{David Baraglia}

\author[L.P.~Schaposnik]{Laura P.~ Schaposnik}
\address{
David Baraglia:  The University of Adelaide, Adelaide SA 5005, Australia}
\email{david.baraglia@adelaide.edu.au}
\address{Laura P.~Schaposnik: University of Illinois at Chicago, 60607 Chicago, USA\\ and FU Berlin, 14195 Berlin, Germany }
\email{schapos@uic.edu}
\date{}

{\abstract{Through Cayley and Langlands type correspondences, we give a geometric description of the moduli spaces of real orthogonal and symplectic Higgs bundles of any signature in  the regular fibres of the Hitchin fibration. As applications of our methods, we complete the  concrete abelianization of   real slices corresponding to all quasi-split real forms,  and describe how extra components emerge naturally from the spectral data point of view.}
}
 
\maketitle
%\tableofcontents
% {\noindent \tiny \color{gray} \tt \gitAuthorIsoDate \hfill
% \gitAbbrevHash}
\vspace{-0.3 in}
%\newpage

%%%%%%%%%%%%%%%%%%%%%%%%%%%%%%%%
%%%%%%%%%%%%%%%%%%%%%%%%%%%%%%%%
%%%%%%%%%%%%%%%%%%%%%%%%%%%%%%%%
%%%%%%%%%%%%%%%%%%%%%%%%%%%%%%%%
%%%%%%%%%%%%%%%%%%%%%%%%%%%%%%%%
%%%%%%%%%%%%%%%%%%%%%%%%%%%%%%%%
%%%%%%%%%%%%%%%%%%%%%%%%%%%%%%%%
%%%%%%%%%%%%%%%%%%%%%%%%%%%%%%%%
%%%%%%%%%%%%%%%%%%%%%%%%%%%%%%%%
%%%%%%%%%%%%%%%%%%%%%%%%%%%%%%%%
%%%%%%%%%%%%%%%%%%%%%%%%%%%%%%%%
%%%%%%%%%%%%%%%%%%%%%%%%%%%%%%%%
%%%%%%%%%%%%%%%%%%%%%%%%%%%%%%%%
%%%%%%%%%%%%%%%%%%%%%%%%%%%%%%%%
%%%%%%%%%%%%%%%%%%%%%%%%%%%%%%%%
%%%%%%%%%%%%%%%%%%%%%%%%%%%%%%%%
%%%%%%%%%%%%%%%%%%%%%%%%%%%%%%%%
%%%%%%%%%%%%%%%%%%%%%%%%%%%%%%%%
%%%%%%%%%%%%%%%%%%%%%%%%%%%%%%%%
%%%%%%%%%%%%%%%%%%%%%%%%%%%%%%%%
%%%%%%%%%%%%%%%%%%%%%%%%%%%%%%%%

\section{Introduction}
The moduli space of surface group representations in a reductive Lie group has long been studied, and through  
  non-abelian Hodge theory, Higgs bundles become a natural holomorphic tool through which to understand them. This paper is dedicated to the study of real orthogonal and symplectic Higgs bundles of any signature on a compact Riemann surface $\Sigma$ of genus $g\geq 2$, and through them, of surface group representations into $SO(p+q,p)$ and $Sp(2p+2q,2p)$. 
  Since most of our results have similar proofs in the symplectic and orthogonal setting, we will mainly focus on the moduli space $\mathcal{M}_{SO(p+q,p)}$ of $SO(p+q,p)$-Higgs bundles. The corresponding results for the symplectic counterparts $Sp(2p+2q,2p)$ follow with only minor modifications (see \autoref{sec:symplectic}). A short review of  Higgs bundles and the Hitchin fibration is given in \autoref{Higgs}. 
\smallbreak

\noindent{\bf Cayley and Langlands type correspondences (\autoref{section2}-\autoref{section4}).} 
We consider the restriction of the orthogonal Hitchin map to $h : \mathcal{M}_{SO(p+q,p)} \to \mathcal{A}_{SO(p+q,p)}$ on the moduli space $\mathcal{M}_{SO(p+q,p)}$ of $SO(p+q,p)$-Higgs bundles\footnote{We consider here $q>0$, since for $q=0$ one recovers the split real form $SO(p,p)$, for which the spectral data of the corresponding Higgs bundles is described through \cite[Theorem 4.12]{thesis}.}. Using our Cayley and Langlands type correspondences, we give a geometric description of the regular fibres $F(a) = h^{-1}(a)$ of the Hitchin map over generic points $a \in \mathcal{A}_{SO(p+q,p)}$ (see \autoref{section4}, and in particular \autoref{teorema1}, for details). More precisely, we identify the fibre $F(a)$ with a   covering of the product of two moduli spaces:
\begin{eqnarray}\label{cay}
\mathcal{M}_{Cay}(a) \times  \mathcal{M}_{Lan}(a).
\end{eqnarray}
The covering in question corresponds to certain {\it  extension data $\tau$} as explained below. The Cayley and Langlands moduli spaces $\mathcal{M}_{Cay}(a)$ and $\mathcal{M}_{Lan}(a)$, and the extension $\tau$ are given as follows:
\begin{itemize}
\vspace{0.1 in}

\item $\mathcal{M}_{Cay}(a)$ is a fibre of the Hitchin map for the moduli space of $K^2$-twisted $GL(p,\mathbb{R})$-Higgs bundles on $\Sigma$, which can be identified with the moduli space of line bundles $L$ of order two in the Jacobian of an associated spectral curve. This $K^2$-twisted $GL(p,\mathbb{R})$-Higgs bundle is related to a maximal $Sp(2p,\mathbb{R})$-Higgs bundle through the Cayley correspondence. The construction of such $K^2$-twisted $GL(p,\mathbb{R})$-Higgs bundle from $SO(p+q,p)$-Higgs bundles is done in \autoref{section2}.
%\smallbreak

\item $\mathcal{M}_{Lan}(a)$ is a moduli space of equivariant $SO(q)$-bundles on an auxiliary double cover \linebreak $\pi_C : C\rightarrow \Sigma$ satisfying a condition over the fixed points (\autoref{section3}). The reconstruction of the $SO(p+q,p)$-Higgs bundle $(E,\Phi)$ from this data involves taking an extension of the form
\[
0 \to V_0 \to E \to F \otimes K^{1/2} \to 0,
\]
where $F$ is the $Sp(2p,\mathbb{R})$-Higgs bundle associated to the Cayley moduli space and $V_0$ is the invariant direct image under $\pi_C : C \to \Sigma$ of the equivariant orthogonal bundle on $C$. In the case $q=1$, this procedure takes us from an $Sp(2p , \mathbb{R})$-Higgs bundle to an $SO(p+1,p)$-Higgs bundle in a way that is related to Langlands duality of the corresponding complex groups $Sp(2p,\mathbb{C})$ and $SO(2p+1,\mathbb{C})$. For this reason, we regard the relation between the original $SO(p+q,p)$-Higgs bundle $(E,\Phi)$ and the equivariant orthogonal bundle $M \in \mathcal{M}_{Lan}(a)$ as a Langlands type correspondence.

%\smallbreak

\item The {\it extension data $\tau$} is given by the extension class defining the above extension. The requirement that $E$ is an $SO(p+q,q)$-Higgs bundle limits the possible choices of extension to take values in a torsor over the group $\mathbb{Z}_2^{4p(g-1)-1}$.
\smallbreak
\end{itemize}

As will be explained in \autoref{section-quadratic}, the space $\mathcal{M}_{Lan}(a)$ of equivariant bundles on $C$ is closely related to moduli spaces of {\it quadratic bundles}, objects that have played an important role when studying Higgs bundles for groups of low rank, and which now we show are fundamental for the analysis of all $SO(p+q,p)$-Higgs bundles in general.  \\
 
 It has been predicted by Guichard and Wienhard \cite[Conjecture 5.6]{anna2} that additional connected components coming from {\it positive representations} (through the notion of $\Theta$-positivity), giving further families of higher Teichm\"uller spaces, appear in the moduli space of surface group representations into $SO(p+q,p)$ for $q>1$. From the   perspective of \autoref{teorema1}, natural candidates for such components are those containing Higgs bundles whose spectral data $(L , M , \tau )$ has the form $( \mathcal{O} , \mathcal{O}^q , \tau )$, as explained in \autoref{extra}.  To prove that this actually gives extra components, the monodromy action \'a la \cite{david,david2} should be taken into consideration as well as the behaviour over singular fibres. On the symplectic side, \autoref{no-extra} addresses the absence of any {\it extra} components in the moduli space of $Sp(2p+2q,2p)$-Higgs bundles. \\
 
\noindent {\bf Characteristic classes (\autoref{section5}).} After introducing the main concepts in \autoref{Higgs},  and describing the {\it spectral data} associated to Higgs bundles in $\mathcal{M}_{SO(p+q,p)}$ leading to \autoref{teorema1} in \autoref{section2}-\autoref{section4}, we study in \autoref{section5} the topological invariants, namely Stiefel-Whitney classes, that can be used to distinguish components of the moduli space of $SO(p+q,p)$-Higgs bundles. \autoref{teorema2} shows how the Stiefel-Whitney classes $\omega_1(W), \omega_2(W)$ and $\omega_2(V)$ of an $SO(p+q,p)$-Higgs bundle can be computed from the Cayley and Langlands type correspondences,  and the extension data $(L , M , \tau)$. In particular, whilst the classes of $W$ are determined purely by the Cayley data $L$, the characteristic class $\omega_2(V)$ depends on all the components of the triple $(L,M,\tau)$. \\

\noindent{\bf Parametrizations of components (\autoref{section6}-\autoref{section-hermitian}).}
From \autoref{teorema1} and \autoref{teorema2} one can see that in general the moduli spaces  $\mathcal{M}_{SO(p+q,p)}$ are parametrized by both abelian (Cayley) and non-abelian (Langlands) data, providing the first examples of real slices of the Hitchin fibration which have such property. This should be compared with the moduli spaces of $G$-Higgs bundles, where $G$ is a split real form which only need abelian data \cite{thesis}, the moduli spaces for $G={\rm SU}(p,p)$, ${\rm SU}(p+1,p)$ which also only need abelian data \cite{peon,umm}, and the moduli spaces for $G=SL(p,\mathbb{H})$, $SO(p,\mathbb{H})$, $Sp(2p,2p)$, which one only need non-abelian data  \cite{nonabelian,thesis}.

%\smallbreak

For particular values of $p$ and $q$, the geometric properties of $SO(p+q,p)$-Higgs bundles and corresponding representations become concrete through the application of \autoref{teorema1}. In this paper we consider some geometric and topological consequences that follow from  \autoref{teorema1}  and \autoref{teorema2} in the following cases:

\begin{itemize}
\vspace{0.1 in}
\item {\it Quasi-split real forms (\autoref{section6})}: We show  in \autoref{teorema4} that the regular fibres of the moduli space of Higgs bundles for the quasi-split real forms $SO(p+2,p)$ admit the structure of an abelian group of the form
\[
Prym(C , \Sigma) \times (\mathbb{Z}_2)^{(4p^2+2p)(g-1)+1}.
\] 
\smallbreak

\item {\it Split real forms (\autoref{section6})}: In the case of the moduli space $\mathcal{M}_{SO(p+1,p)}$, the existence of extra components as per \cite[Conjecture 5.6]{anna2} is known to  be true \cite{aparicio,brian2}. We show here  that these extra components emerge naturally from the the extension class $\tau$, as suggested in \cite{ort} and  shown in \autoref{teorema5}. Therefore, our methods provide a simple conceptual explanation for the existence of these components. Furthermore, using our spectral data constructions, we are able to write down explicit parametrizations of the Higgs bundles in these components.
\smallbreak

The above has implications for $q>1$: since these components can be seen through the spectral data, we comment in \autoref{extra} on the extra components that may appear for $q>1$. \smallbreak

\item {\it Groups of Hermitian type (\autoref{section-hermitian})}: The study of surface group representations into Hermitian groups of rank 2 reduces to considering the group $SO(2+q,2)$  \cite{anna1}. Using \autoref{teorema1} and \autoref{teorema2}, we show that the Cayley data is parametrized by fibres of the moduli space $\mathcal{M}_{Sp(4,\mathbb{R})}^{max}$ of maximal $Sp(4,\mathbb{R})$ surface representations. In \autoref{section-hermitian} we consider the implications for $SO(2+q,2)$-Higgs bundles of the extra components in $\mathcal{M}_{Sp(4,\mathbb{R})}^{max}$, obtained by Gothen   \cite{gothen}. \smallbreak
\end{itemize}

\noindent{\bf Some further notes on Langlands duality (\autoref{section-structures}).} 
The moduli space $\mathcal{M}_{SO(p+q,p)}$ of $SO(p+q,p)$-Higgs bundles can be thought of as a $(B,A,A)$-brane in the moduli space $\mathcal{M}_{SO(2p+q,\mathbb{C})}$ of $SO(2p+q,\mathbb{C})$-Higgs bundles. According to Langlands duality, interpreted as mirror symmetry between the moduli spaces of Higgs bundles for Langlands dual groups, the mirror of $\mathcal{M}_{SO(p+q,p)}$ should be a $(B,B,B)$-brane in the moduli space for the Langlands dual group of $SO(2p+q,\mathbb{C})$. In \autoref{section-structures}, building on our previous work, we give  a conjectural description of the dual $(B,B,B)$-brane. In particular, we conjecture that the underlying support of the brane depends only on $p$, while the moduli space in which the brane is embedded depends on both $p$ and $q$.
 \smallbreak
\noindent{\bf Acknowledgments}. The authors are thankful for financial support from  from U.S. National Science Foundation grants DMS 1107452, 1107263, 1107367 "RNMS: GEometric structures And Representation varieties" (the GEAR Network) which financed several research visits during which the paper was written. D.~Baraglia is financially supported by the Australian Research Council Discovery Early Career
Researcher Award DE160100024.
L.P.~Schaposnik is partially supported by the NSF grant DMS-1509693, and by the Alexander von Humboldt Foundation.

%%%%%%%%%%%%%%%%%%%%%%%%%%%%%%%%
%%%%%%%%%%%%%%%%%%%%%%%%%%%%%%%%
%%%%%%%%%%%%%%%%%%%%%%%%%%%%%%%%
%%%%%%%%%%%%%%%%%%%%%%%%%%%%%%%%
%%%%%%%%%%%%%%%%%%%%%%%%%%%%%%%%
%%%%%%%%%%%%%%%%%%%%%%%%%%%%%%%%
%%%%%%%%%%%%%%%%%%%%%%%%%%%%%%%%
%%%%%%%%%%%%%%%%%%%%%%%%%%%%%%%%
%%%%%%%%%%%%%%%%%%%%%%%%%%%%%%%%
%%%%%%%%%%%%%%%%%%%%%%%%%%%%%%%%
%%%%%%%%%%%%%%%%%%%%%%%%%%%%%%%%
%%%%%%%%%%%%%%%%%%%%%%%%%%%%%%%%
%%%%%%%%%%%%%%%%%%%%%%%%%%%%%%%%
%%%%%%%%%%%%%%%%%%%%%%%%%%%%%%%%
%%%%%%%%%%%%%%%%%%%%%%%%%%%%%%%%
%%%%%%%%%%%%%%%%%%%%%%%%%%%%%%%%
%%%%%%%%%%%%%%%%%%%%%%%%%%%%%%%%
%%%%%%%%%%%%%%%%%%%%%%%%%%%%%%%%
%%%%%%%%%%%%%%%%%%%%%%%%%%%%%%%%
%%%%%%%%%%%%%%%%%%%%%%%%%%%%%%%%
%%%%%%%%%%%%%%%%%%%%%%%%%%%%%%%%
\section{Higgs bundles and the Hitchin fibration}\label{Higgs}
  
Throughout the paper we will consider a compact Riemann surface $\Sigma$ of genus $g\geq 2$ with canonical bundle $K=T^*\Sigma$. In what follows, we recall some of the main properties of complex and real Higgs bundles, as well as the associated Hitchin fibration. 
  
\subsection{Higgs bundles for complex and real groups}  \label{definitions}
We begin by briefly reviewing the notions of Higgs bundles for real and complex groups  which are relevant to this paper. Further details can be found in standard references such as Hitchin \cite{N1,N2} and Simpson \cite{S1,S2,S3}. Higgs bundles on $\Sigma$ are pairs $(E,\Phi)$ where  \begin{itemize}
 \item $E$   is a holomorphic vector bundle on $\Sigma$,
 \item    the Higgs field $\Phi: E\rightarrow E \otimes K$, is a holomorphic $K$-valued endomorphism.
  \end{itemize}
More generally, for a complex reductive Lie group $\GC$, we have the following \cite{N2}.

\begin{Definition} A $\GC$-Higgs bundle is a pair $(P,\Phi)$, where $P$ is a holomorphic principal $\GC$ bundle, and $\Phi$ is a holomorphic section of ${\rm ad}(P)\otimes K$, where ${\rm ad}(P)$ is the adjoint bundle of $P$. \label{def1}
\end{Definition}

Higgs bundles were introduced by Hitchin in \cite{N1} as solutions of the so-called  {\it Hitchin equations} \begin{eqnarray}
F_A+ [\Phi,\Phi^*]=0, ~{~}~ ~\overline{\partial}_{A}\Phi=0,\label{equation}\end{eqnarray}    where $F_A$ is the curvature of the unitary connection $\nabla_A=\partial_{A}+\overline{\partial}_{A}$ associated to a reduction of structure of $P$ to the maximal compact subgroup of $\GC$. One can construct the moduli space $\MGC$ of solutions to the $\GC$-Hitchin equations, which admits a natural hyperk\"ahler metric over its smooth points. By the work of Hitchin and Simpson, when $\GC$ is semisimple, the existence of a unitary connection satisfying the Hitchin equations is equivalent to polystability of the pair $(P , \Phi )$. From this work it also follows that $\MGC$ can also be identified with the moduli space of polystable $\GC$-Higgs bundles. When $\GC$ is reductive but not semisimple (e.g., $GL(n,\mathbb{C})$) we will simply take $\MGC$ to be the moduli space of polystable $\GC$-Higgs bundles.

Given a real form $G$ of the complex reductive lie group $\GC$, we may define $G$-Higgs bundles as follows. Let $H$ be the maximal compact subgroup of $G$ and consider the Cartan decomposition $\mathfrak{g} = \mathfrak{h}\oplus \mathfrak{m}$ of $\mathfrak{g}$, where $\mathfrak{h}$ is the Lie algebra of $H$, and $\mathfrak{m}$ its orthogonal complement. 
 This induces a decomposition of the Lie algebra $\mathfrak{g}_\C=\mathfrak{h}^{\mathbb{C}}\oplus \mathfrak{m}^{\mathbb{C}}$ of $G_\C$. Note that the Lie algebras satisfy
$ [\mathfrak{h}, \mathfrak{h}]\subset\mathfrak{h}$, $[\mathfrak{h,\mathfrak{m}}]\subset\mathfrak{m}$, $[\mathfrak{m},\mathfrak{m}]\subset \mathfrak{h},$  and there is an induced isotropy representation ${\rm Ad}|_{H^{\mathbb{C}}}: H^{\mathbb{C}}\rightarrow GL(\mathfrak{m}^{\mathbb{C}})$.

\begin{Definition}\label{def2}
 A {\rm principal} $G${\rm -Higgs bundle } is a pair $(P,\Phi)$ where
\begin{itemize}
 \item $P$ is a holomorphic principal $H^{\mathbb{C}}$-bundle on $\Sigma$,
 \item $\Phi$ is a holomorphic section of $P\times_{Ad}\mathfrak{m}^{\mathbb{C}}\otimes K$.
\end{itemize}
\end{Definition}

Similar to the case of Higgs bundles for complex groups, there are notions of stability, semistability and polystability for $G$-Higgs bundles, and one can see that the polystability of a $G$-Higgs bundle for $G\subset GL(n,\mathbb{C})$ is equivalent to the polystability of the corresponding $GL(n,\mathbb{C})$-Higgs bundle. However, a $G$-Higgs bundle can be stable as a $G$-Higgs bundle but not as a $GL(n,\mathbb{C})$-Higgs bundle. We denote by $\mathcal{M}_{G}$ the moduli space of polystable $G$-Higgs bundles on   $\Sigma$.

\subsection{Spectral data and the Hitchin fibration} \label{fibration}  A natural way of studying the moduli spaces  $\mathcal{M}_{\GC}$ of $ \GC$-Higgs bundles is to use the Hitchin fibration  \cite{N2}.   Let $\{p_{1}, \ldots, p_k\}$ be a homogeneous basis for the algebra of invariant polynomials on the Lie algebra  $\mathfrak{g}_{c}$ of $ \GC$, and let $d_{i}$ denote the degree of $p_i$. The {\it Hitchin fibration} is then given by
\begin{eqnarray} h~:~ \mathcal{M}_{ \GC}&\longrightarrow&\mathcal{A}_{ \GC}:=\bigoplus_{i=1}^{k}H^{0}(\Sigma,K^{d_{i}}),
\end{eqnarray} 
 where $h:(E,\Phi)\mapsto (p_{1}(\Phi), \ldots, p_{k}(\Phi))$ is referred to as the {\it Hitchin~map}: it is a proper map for any choice of basis and makes  the  moduli space into an integrable system \cite{N2}.

Each connected component of a generic fibre of the Hitchin map is an abelian variety. In the case of $GL(n,\mathbb{C})$-Higgs bundles this can be seen using spectral data \cite{N2,BNR}. A $GL(n,\mathbb{C})$-Higgs bundle  $(E,\Phi)$ defines an algebraic curve, called the {\em spectral curve} of $(E,\Phi)$:
 \begin{eqnarray}S=\{{\rm det}(\eta I -\Phi)=0\} \subset {\rm Tot}(K), \end{eqnarray}
where ${\rm Tot}(K)$ is the total space of $K$ and $\eta$ is the tautological section of $K$ on ${\rm Tot}(K)$. We say that $(E,\Phi)$ lies in the {\em regular locus} of $\mathcal{M}_{GL(n,\mathbb{C})}$ if the curve $S$ is non-singular. Let $\pi : S \to \Sigma$ denote the natural projection to $\Sigma$ and let $\eta \in H^0( S , \pi^*(K) )$ denote the restriction of the tautological section of $K$ to $S$. If $(E,\Phi)$ is in the regular locus then there exists a line bundle $L \to S$ for which $E = \pi_* L$ and $\Phi$ is obtained by pushing down the map $\eta : L \to L \otimes \pi^*(K)$. In this way, we recover the pair $(E,\Phi)$ from the pair $(S,L)$. We call $(S , L)$ the {\em spectral data} associated to the pair $(E , \Phi)$.

Note that the spectral curve $S$ of the pair $(E , \Phi)$ depends only on  the characteristic polynomial of $\Phi$ and hence depends only on the image of $(E,\Phi)$ under the Hitchin map. In this way, we can associate a spectral curve $S$ to any point $a \in \mathcal{A}_{ \GL(n,\mathbb{C})}$ in the base of the Hitchin system. If $a$ is in the regular locus of $\mathcal{A}_{\GL(n,\mathbb{C})}$, in other words, if the associated spectral curve $S$ is smooth, then the spectral data construction identifies the fibre $h^{-1}(a)$ of the Hitchin system with $Pic(S)$, the Picard variety of the spectral curve $S$. The connected components of $Pic(S)$ are, of course, isomorphic to copies of $Jac(S)$, the Jacobian of $S$. In particular this confirms that the components of the regular fibres are abelian varieties.
  
  \subsection{Complex orthogonal Higgs bundles} Since the core of this paper is on the geometry of the moduli space of orthogonal Higgs bundles, we shall give here a thorough description of these objects and their spectral data. 
From  \autoref{def1}, an $SO(2p+q,\mathbb{C})$-Higgs bundle consists of a pair $(E,\Phi)$ where 
\begin{enumerate}
\item{$E$ is a holomorphic vector bundle of rank $2p+q$ with a non-degenerate symmetric bilinear form $(v,w)$, together with a trivialization of the determinant bundle $\Lambda^{2p+q}E$ as a $\mathbb{Z}_2$-line bundle (i.e. a trivialisation of the principal $\mathbb{Z}_2$-bundle to which $\Lambda^{2p+q} E$ is associated).}
\item{$\Phi\in H^{0}(\Sigma,{\rm End}(E)\otimes K)$ is a Higgs field which satisfies $(\Phi v,w)=-(v,\Phi w)$.}
\end{enumerate}

We denote by $\mathcal{M}_{SO(2p+q,\mathbb{C})}$ the moduli space of $S$-equivalence classes of semi-stable $SO(2p+q,\mathbb{C})$-Higgs bundles. This moduli space has two connected components, labeled  by the second Stiefel-Whitney class $w_{2}(E) \in H^{2}(\Sigma,\mathbb{Z}_{2}) \cong  \mathbb{Z}_{2}$, depending on whether $E$ has a lift to a spin bundle or not \cite{N5}. 

  \subsection{$SO(p+q,p)$-Higgs bundles}\label{orto} From \autoref{def2}, an $SO(p+q,p)$-Higgs bundle  consists of:
\begin{enumerate}
\item{A rank $p+q$ orthogonal bundle $(V,Q_V)$}
\item{A rank $p$ orthogonal bundle $(W,Q_W)$}
\item{A holomorphic bundle map $\beta : W \to V \otimes K$}
\item{An isomorphism $\det(V) \cong \det(W)$ as $\mathbb{Z}_2$-line bundles.}
\end{enumerate}

Given an $SO(p+q,p)$-Higgs bundle $(V,W , \beta)$, the associated $SO(2p+q,\mathbb{C})$-Higgs bundle $(E,\Phi)$ is obtained by setting $E=V\oplus W$ with bilinear form
\[
( (x,y) , (x',y') ) = Q_V(x,x') - Q_W(y,y')
\]
and Higgs field $\Phi:E\rightarrow E\otimes K$   given by
\begin{eqnarray}\Phi=\left(\begin{array}{cc}
              0&\beta\\
\gamma&0
             \end{array} \right),\label{elphi}
\end{eqnarray} 
where $\gamma = \beta^t$ is the orthogonal transpose of $\beta$,   obtained using the orthogonal structures on $V, W$. In the moduli space of $SO(2p+q,\mathbb{C})$ Higgs bundles,   $SO (p,p+q)$ Higgs bundles are fixed points of the involution 
\[\Theta:~(E,\Phi)\mapsto (E, -\Phi)\]corresponding to pairs $(E,\Phi)$ such that there is an isomorphism $f : (E,\Phi) \to (E,-\Phi)$ induced by an involution $f : E \to E$ whose $+1$ and $-1$-eigenspaces have dimensions $p+q$ and $p$ respectively.

The curve defined by the characteristic equation of the Higgs field $\Phi$  is a reducible curve: an $SO(p+q,p)$-Higgs field $\Phi$ always has a zero eigenspace of dimension $\ge q$ since its characteristic polynomial is of the form
\begin{eqnarray}
{\rm det}(\eta-\Phi)= \eta^q(\eta^{2p}+a_{1}\eta^{2p-2}+\ldots +a_{p-1}\eta^{2}+a_{p}).\label{poli}\end{eqnarray}
 In the case of  $q=1$ it is shown in   \cite[Section 4.1]{langlands} that the zero eigenspace $E_{0}$ is given by $E_{0}\cong K^{-2p}$. We will see in   \autoref{sec:quadratic} how a similar characterisation of the zero eigenspace can be made for any $q$ in terms of {\it quadratic bundles}.

The generically irreducible component of the  characteristic polynomial Eq.~\eqref{poli}  defines an associated 2p-fold spectral curve  $\pi:S\rightarrow \Sigma$  whose  equation is
\begin{eqnarray}\eta^{2p}+a_{1}\eta^{2p-2}+\ldots +a_{p-1}\eta^{2}+a_{p}=0,\label{curvepp}\end{eqnarray}
where $a_{i}\in H^{0}(\Sigma, K^{i})$.  By Bertini's theorem, this is a generically smooth curve. In this paper, we are mostly concerned with spectral curves satisfying the following conditions:
\begin{assumption}\label{ass1}
Assume that $S$ is smooth and that $a_{p} , a_{p-1}$ do not simultaneously vanish.
\end{assumption}

The  curve $S$ has an involution $\sigma$ which acts as $\sigma(\eta)=-\eta$. Thus, we may consider the quotient curve $\overline{S}=S/\sigma$ in the total space of $K^{2}$,  for which $S$ is a double cover
$\rho: S\rightarrow \overline{S}$, leading to the following diagram
 \begin{eqnarray}
 \xymatrix{ S\ar[rd]^{\pi}_{2p:1}\ar[rr]_{\rho}^{2:1}&&\bar S\ar[ld]^{p:1}_{\bar \pi}\\
 &\Sigma&
 }\label{hola}
 \end{eqnarray}
  The covers $S$ and $\bar S$ have, respectively,  genera 
\begin{equation*}
\begin{aligned}
g_{S}&= 1+4p^{2}(g-1),\\
g_{\bar S}&=(2p^2-p)(g-1)+1.
\end{aligned}
\end{equation*}
By the adjunction formula, the canonical bundles of these covers are, respectively, 
$K_S=\pi^*K^{2p}$ and
$K_{\bar S}=\bar \pi^*K^{2p-1}.
$ 
For $\xi=\eta^2$ the tautological section of $K^2$, the quotient curve is given by 
\begin{eqnarray}\bar S=\{\xi^{p}+a_{1}\xi^{p-1}+\ldots +a_{p-1}\xi+a_{p}=0\}\subset Tot(K^2) \label{curvep}\end{eqnarray}
and will become a key ingredient when constructing vector bundles associated to those Higgs bundles with signature. 

Note that smoothness of $S$ is equivalent to smoothness of $\overline{S}$ and that $a_{p}$ has only simple zeros. If $S$ is smooth, the condition that $a_p$ and $a_{p-1}$ do not simultaneously vanish is equivalent to requiring that the fixed points of $\sigma$ are simple branch points of $\pi : S \to \Sigma$.

Let $\mathcal{M}_{SO(p+q,p)}$ denote the moduli space of semistable $SO(p+q,p)$-Higgs bundles (see \cite{aparicio} for the construction of these spaces). The restriction of the Hitchin map to the real moduli space can be identified with the map
\[
h : \mathcal{M}_{SO(p+q,p)} \to \mathcal{A}_{SO(p+q,p)} = \bigoplus_{i=1}^p H^0( \Sigma , K^{2i}), \quad \quad h(E,\Phi) = (a_1 , a_2 , \dots , a_p),
\]
where $a_1 , \dots , a_p$ are as in \eqref{poli}. Let $\mathcal{A}^{\rm reg}_{SO(p+q,p)}$ denote the locus of points $a \in \mathcal{A}^{\rm reg}_{SO(p+q,p)}$ satisfying \autoref{ass1} and let $\mathcal{M}^{\rm reg}_{SO(p+q,p)}$ denote the pre-image of $\mathcal{A}^{\rm reg}_{SO(p+q,p)}$, so that $h$ restricts to a map $h : \mathcal{M}^{\rm reg}_{SO(p+q,p)} \to \mathcal{A}^{\rm reg}_{SO(p+q,p)}$.

\section{A Cayley type correspondence:\\ Associated $K^2$-twisted ${\rm GL}(p,\mathbb{R})$-Higgs bundles}\label{section2}
Cayley correspondences for Higgs bundles have long been studied, and they provide a procedure in which one can obtain a correspondence between the moduli space of Higgs bundles for a  Hermitian group of tube type, and the moduli space of $K^2$-twisted Higgs bundles for a certain associated group. The interested reader should refer to \cite{Roberto} and references therein for further details of these correspondences. 

In what follows we will construct a naturally defined $K^2$-twisted $GL(p , \mathbb{R})$-Higgs bundle associated to $SO(p+q,p)$-Higgs bundles $(E=V\oplus W, \Phi)$ satisfying \autoref{ass1} (for the case $q=1$, c.f. \cite{ort}), providing a Cayley-type correspondence.  Note that, as mentioned in \eqref{cay}, our aim is to understand the regular fibres of the moduli space $\mathcal{M}_{SO(p+q,p)}$ in terms of two associated moduli spaces $\mathcal{M}_{Cay}, \mathcal{M}_{Lan}$. Only when the group $SO(p+q,p)$ is of Hermitian type (i.e., only for $p=2$), one recovers the standard Cayley correspondence between $SO_0(2+q,2)$-Higgs bundles and $K^2$-twisted $SO_0(1,1)\times SO(1,q+1)$-Higgs bundles (for the latter, see \cite[Table C.4.]{Roberto}). This case will be addressed in \autoref{section-hermitian} where the moduli spaces $\mathcal{M}_{SO(2+q,2)}$ are studied in further detail. 

\subsection{Unitary structure} \label{uni}
Under   \autoref{ass1}, we have that $a_{p}$ is not identically zero. This means that $\Phi$ generically has rank $2p$ and thus $\beta$ and $\gamma$ both generically have rank $p$. So $\beta$ is generically injective and $\gamma$ generically has a $q$-dimensional kernel. We will see there is a canonically induced $U(p,p)$-Higgs bundles which can be obtained by considering the induced bundle
\begin{eqnarray}
&V_{0}:={\rm ker}(\gamma)&~{\rm where~}~ \gamma: V \rightarrow W\otimes K. \label{W1}
\end{eqnarray}
By this, we mean that $\mathcal{O}(V_0)$ is the kernel of $\gamma : \mathcal{O}(V) \to \mathcal{O}(W \otimes K)$,  defining  a subbundle $V_0 \subset V$.

\begin{Proposition} \label{E+}
Let $V_{1}:=V/V_0$. We obtain an induced $U(p,p)$-Higgs bundle given by:
 \begin{eqnarray}
 &(E_+:=   V_1\oplus  W, \Phi_+), \label{H1}
  \end{eqnarray}
where $\Phi_+$ is determined by the following commutative diagram
\[
\xymatrix{
V \oplus W \ar[d] \ar[r]^-{\Phi} & V \oplus W \ar[d] \\
V_1 \oplus W \ar[r]^-{\Phi_+} & V_1 \oplus W.
}
\]
\end{Proposition}
  \begin{proof}
By the definition of $V_0$,  the map $\gamma$ factors as (c.f. \cite[\textsection 4]{NS}):
\begin{equation}\xymatrix{
0 \ar[r] & V_0 \ar[r] & V \ar[r] \ar[d]^-\gamma & V_1 \ar[r] \ar[dl]^-{\gamma_+} & 0 \\
& & W \otimes K & &
}\label{D1}
\end{equation}
where the top row is a short exact sequence of vector bundles and $\gamma_+ : V_1 \to W \otimes K$ is generically an isomorphism. Dually, we obtain:
\begin{equation}\xymatrix{
0 & V_0^* \otimes K \ar[l] & V \otimes K \ar[l] & V_1^* \otimes K \ar[l] & 0 \ar[l] \\
& & W \ar[u]^-\beta \ar[ur]_-{\gamma_+^t} & &
}\label{D2}
\end{equation}
Define $\beta_+ : W \to V_1 \otimes K$ by the following diagram:
\begin{equation}\xymatrix{
0 \ar[r] & V_0 \otimes K \ar[r] & V \otimes K \ar[r] & V_1 \otimes K \ar[r] & 0, \\
& & W \ar[u]^-\beta \ar[ur]_-{\beta_+} & &
}\label{D3}
\end{equation}
and note that there is a dual diagram:
\begin{equation}\xymatrix{
0 & V_0^* \ar[l] & V \ar[l] \ar[d]^-\gamma & V_1^* \ar[l] \ar[dl]^-{\beta^t_+} & 0 \ar[l] \\
& & W \otimes K & &
}\label{D4}
\end{equation}
From \eqref{D1} and \eqref{D2} there are bundle maps $\gamma_+ : V_1 \to W \otimes K$ and $\beta_+ : W \to V_1 \otimes K$. The data $(V,W,\beta_+ , \gamma_+)$ defines a $U(p,p)$-Higgs bundle \cite{umm}. The underlying $GL(2p,\mathbb{C})$-Higgs bundle of this $U(p,p)$-Higgs bundle from Eq.~\eqref{H1} is given by:
\begin{equation}
E_+ = V_1 \oplus W, \quad \quad \Phi_+ = \left( \begin{matrix} 0 & \beta_+ \\ \gamma_+ & 0 \end{matrix} \right).\label{gl2p}
\end{equation}
To finish the proof, we just have to note that $\Phi_+$ as defined by \eqref{gl2p} agrees with $\Phi_+$ as given in the statement of the proposition. This follows easily from the above commutative diagrams.
\end{proof}

\begin{Remark}
From the construction of the $U(p,p)$-Higgs bundle $(E_+:= V_1\oplus W , \Phi_{+})$, we have that 
\begin{equation*}
\det( \eta - \Phi_+) = \eta^{2p} + a_1 \eta^{p-1} + \dots + a_{p},
\end{equation*}
and thus the 2p-fold cover $\pi:S\rightarrow \Sigma$ is in fact the spectral curve of $(E_+ , \Phi_+)$.
\end{Remark} 

\subsection{Symplectic structure} In what follows we show that the $U(p,p)$-Higgs bundle defined in \autoref{uni} give rise to a real symplectic Higgs bundle. 

\begin{Lemma}\label{1hola}
Let $( E , \Phi )$ be a $GL(n,\mathbb{C})$-Higgs bundle and suppose that $det( \Phi )$ vanishes to first order at $x \in \Sigma$. Then $Ker( \Phi_x )$ is $1$-dimensional.
\end{Lemma}
\begin{proof}
Suppose on the contrary that $\dim(Ker( \Phi_x )) \ge 2$. In such case there exists linearly independent $e_1 , e_2 \in E_x$ with $\Phi_x(e_1) = \Phi_x(e_2) = 0$. Extending $e_1,e_2$ to a basis $e_1 , e_2 , \dots e_n$ of $E_x$,  one can choose  a local frame $\tilde{e}_1 , \tilde{e}_2 , \dots , \tilde{e}_n$ for $E$ with $\tilde{e}_j(x) = e_j$ and  a local trivialisation of $K$. Then $\Phi_x( \tilde{e}_j ) = \sum_{i=1}^{n}b_{ij}(x) \tilde{e}_i(x)$ for some holomorphic functions $b_{ij}(x)$ on the Riemann surface $\Sigma$. Since $\Phi_x( e_j ) = 0$ for $j=1,2$, we have that $b_{ij}(x) = 0$ for $j = 1,2$. Therefore two columns of the matrix $[ b_{ij} ]$ vanish at $x$ and so $\det [ b_{ij} ] = \det( \Phi )$ vanishes to at least second order at $x$, a contradiction.
\end{proof}

\begin{Lemma}\label{gammaisom} The map $\gamma_+ : V_1 \to W \otimes K$ is an isomorphism.
\end{Lemma}
\begin{proof}By   \autoref{1hola}, over each $x\in \Sigma$ the Higgs field $\Phi_+$ from Eq.~\eqref{H1} satisfies the following:
\begin{itemize}
\item{If $a_{p}(x) \neq 0$, then $Ker(\Phi_+) = \{ 0 \}.$}
\item{If $a_{p}(x) = 0$, then $Ker(\Phi_+)$ is $1$-dimensional.}
\end{itemize}
It follows that at each zero of $a_{p}$, either $\beta_+$ or $\gamma_+$ has a $1$-dimensional kernel, and the other is injective. But from its definition, we see that $\beta_+$ factors as:
$\xymatrix{
W \ar[r]^-{\gamma_+^t} & V_1^* \otimes K \ar[r] & V \otimes K \ar[r] & V_1 \otimes K.
}
$ Therefore if $\gamma_+$ has a kernel, so does $\beta_+$. It follows that $\gamma_+$ is necessarily everywhere injective and hence $\gamma_+ : V_1 \to W \otimes K$ is an isomorphism.
\end{proof}
\begin{Proposition}\label{sim}
For each choice of $K^{1/2}$, the Higgs bundle $(E_+ , \Phi_+)$ from \eqref{gl2p} defines a canonical $Sp(2p, \R)$-Higgs bundle $(F,\Phi_F)$ with maximal Toledo invariant, where
\[F = (W \otimes K^{1/2}) \oplus (W \otimes K^{-1/2}), \quad \quad \Phi_F = \left( \begin{matrix} 0 &\gamma_+ \circ \beta_+\\ Id & 0 \end{matrix} \right).
\]
\end{Proposition}

\begin{proof}
Recall that in Eq.~\eqref{gl2p} we   constructed the $GL(2p , \mathbb{C})$-Higgs bundle $(E_+ , \Phi_+)$. Let $K^{1/2}$ be a fixed choice of square root of $K$. We will now tensor $(E_+ , \Phi_+)$ by $K^{-1/2}$ to obtain a Higgs bundle $(F = E_+ \otimes K^{-1/2} , \Phi_F = \Phi_+ \otimes Id )$. Using $\gamma_+$ to identify $V_1$ with $W \otimes K$, we have:
\begin{equation}
F = (W \otimes K^{1/2}) \oplus (W \otimes K^{-1/2}), \quad \quad \Phi_F = \left( \begin{matrix} 0 & \beta_F \\ Id & 0 \end{matrix} \right)\label{phif}
\end{equation}
where   $\beta_F:=\gamma_+ \circ \beta_+ : W \to V_1 \otimes K \cong W \otimes K^2$.
Note that the orthogonal structure $Q_W$ of $W$ satisfies   $Q_W( \beta_F a , b ) = Q_W( a , \beta_F b)$: indeed,    by the   diagrams in \eqref{D1}-\eqref{D4} the map $\beta_F$ can be written as 
$W \buildrel \beta \over \longrightarrow V \otimes K \buildrel \gamma \over \longrightarrow W \otimes K^2$.
Then since $\gamma = \beta^t$, we have that $(\gamma \beta)^t = \beta^t \gamma^t = \gamma \beta$. Using the orthogonal structure $Q_W$, one can make $(F,\Phi_F)$ into an $Sp(2p, \mathbb{R})$-Higgs bundle as follows. The symplectic form $\omega_F$ is defined to be:
\begin{equation*}
\omega_F( (a,b) , (c,d) ) = Q_W(a,d) - Q_W(b,c).
\end{equation*}
It is straightforward to check that $\omega_F( \Phi_F u , v) = -\omega_F( u , \Phi_F v)$, and so from \autoref{def1} we have that $(F,\Phi_F)$ is an $Sp(2p , \mathbb{C})$-Higgs bundle. In fact, this is an $Sp(2p, \mathbb{R})$-Higgs bundle since $F$ splits into a sum $F = (W \otimes K^{1/2}) \oplus (W \otimes K^{-1/2})$ of Lagrangian subbundles and $\Phi_F$ is off-diagonal with respect to this splitting. Moreover, since $\Phi_F$ is as in \eqref{phif}, this means that $(F , \Phi_F)$ is an $Sp(2p , \mathbb{R})$-Higgs bundle with maximal Toledo invariant.
\end{proof}

\subsection{Symplectic Cayley correspondence}\label{Cay} Under the Cayley correspondence, maximal $Sp(2p  , \mathbb{R})$-Higgs bundles correspond to $K^2$-twisted $GL(p , \mathbb{R})$-Higgs bundles \cite{Roberto}. Recall that a $K^2$-twisted $GL(p , \mathbb{R})$-Higgs bundle is a triple  $(W , Q_W , \beta)$ consisting of a rank $p$ orthogonal vector bundle $(W , Q_W)$ and a holomorphic   map $\beta : W \to W \otimes K^2$ such that   $Q_W( \beta a , b) = Q_W( a , \beta b)$. Therefore, in our case the $K^2$-twisted $GL(p,\mathbb{R})$-Higgs bundle determined by $(F , \Phi_F)$ is exactly the triple $(W , Q_W , \beta_F)$.

 The construction of the $Sp(2p,\mathbb{R})$-Higgs bundle $(F , \Phi_F)$ from $(V , W , \beta)$ involved the choice of a square root of $K$. Similarly, the Cayley correspondence relating $(F , \Phi_F)$ to $(W , \beta_F)$ requires a choice of square root of $K$. In the discussion above we have chosen the same square root in both instances. With this convention in place, one sees that the construction of $(W , \beta_F)$ from $(V , W , \beta)$ does not depend on the choice of this square root, so $(W , \beta_F)$ is canonically associated to $(V  ,W , \beta)$. 
 
\begin{Remark}
One should note that the Cayley partner can be viewed directly from the $U(p,p)$-Higgs bundle $(E_+,\Phi_+)$ for a choice of square root of $K$, as described in \cite[Remark 3.7]{umm}. In particular,  the spectral curve of the $K^2$-twisted $GL(p,\mathbb{R})$-Higgs bundle determined by $(F , \Phi_F)$ is given by $\overline S$ as in Eq.~\eqref{hola}.
\end{Remark}
  
 The $K^2$-twisted $GL(p , \mathbb{R})$-Higgs bundles 
associated to the maximal $Sp(2p  , \mathbb{R})$-Higgs bundles obtained from the $SO(p+q,p)$-Higgs pairs above can be described also in terms of spectral data. Indeed, since $GL(p , \mathbb{R})$ is the split real form of $GL(p,\C)$, from \cite[Theorem 4.12]{thesis} we have that these Higgs bundles over a smooth spectral curve $\bar S$ in the regular locus of the Hitchin fibration are given by set of $2$-torsion points in the Jacobian of $\bar S$:
\begin{eqnarray}
{\rm Jac}(\bar S)[2]:=\{L\in {\rm Jac}(\bar S)~|~L^2\cong \mathcal{O} \}. \label{spectralupp}\label{jac}
\end{eqnarray}

As will be recalled in \autoref{section4}, the relation between $W$ and $L$ is that $W = \overline{\pi}_*( L \otimes \overline{\pi}^* K^{(p-1)})$. Since $L$ has order $2$, it can be viewed as a line bundle with orthogonal structure. Then $W$ inherits an orthogonal structure by relative duality.

\section{A Langlands type correspondence:\\
Quadratic bundles and the auxiliary spectral curve}\label{sec:quadratic}\label{section3}
 
Recall that in \autoref{section2}, the moduli spaces of $SO(p+q,p)$-Higgs bundles were shown to have an associated symplectic $Sp(2p,\C)$-Higgs bundle. In this and the following section we will study in the spirit of Langlands duality, the problem of reconstructing the $SO(2p+q,\mathbb{C})$-Higgs bundle starting from $Sp(2p,\mathbb{C})$. In the case $q=1$, the group $SO(2p+1,\mathbb{C})$ is the Langlands dual of $Sp(2p,\mathbb{C})$ and as shown by Hitchin \cite{langlands}, the process of reconstructing the $SO(2p+1,\mathbb{C})$-Higgs bundle exhibits the duality between the fibres of the $Sp(2p,\mathbb{C})$ and $SO(2p+1,\mathbb{C})$-Higgs bundle moduli spaces. For $q>1$, we refer to the relation between $Sp(2p,\mathbb{C})$ and $SO(2p+q,\mathbb{C})$-Higgs bundles as a {\em Langlands type correspondence}.
 
We will first show in \autoref{section-quadratic} that part of the data needed to recover the orthogonal Higgs bundle is parametrized by the moduli space of quadratic bundles $(V_0,Q_0)$  with fixed determinant (the relevant definitions are given in \autoref{section-quadratic}). We will then show that this space can actually be identified with a certain moduli space $\mathcal{M}_{Lan}$ of equivariant orthogonal bundles on an auxiliary spectral curve. This is shown in \autoref{quad1}, by considering the auxiliary double cover
\begin{equation}
C = \{ \zeta^2 - a_p=0\} \subset Tot(K^p), \label{curve}
\end{equation}
where $\zeta$ is the tautological section of $K^p$ on $Tot(K^p)$. These spaces $\mathcal{M}_{Lan}$ will become fundamental when describing the nonabelianization of orthogonal Higgs bundles in \autoref{section4}, and thus in what follows, we consider some general properties of quadratic bundles in order to later understand the ones arising from Higgs bundles.

 \subsection{Quadratic bundles and double covers}\label{section-quadratic} 
 Understanding how vector bundles on Riemann surfaces can be obtained as direct images of bundles on coverings has been of interest for many decades, and the question is closely related to Higgs bundles (e.g., see \cite{BNR}). In what follows we describe how  {\it quadratic bundle} on $\Sigma$ can be obtained naturally using double covers of $\Sigma$. 
\begin{Definition}
A {\it quadratic bundle} is a pair $(V_0 , Q_0)$, where $V_0$ is a holomorphic vector bundle and $Q_0$ is a holomorphic section of $Sym^2(V_0^*)$. 
\end{Definition}
Given $(V_0 , Q_0)$, we may view $Q_0$ as a map $Q_0 : V_0 \to V_0^*$ and take its induced determinant map $\det(Q_0) : \det(V_0) \to \det(V_0^*)$. Thus, $\det(Q_0)$ may be regarded as a section of $\det(V_0)^{-2}$.

\begin{Definition}
We say that $(V_0 , Q_0)$ is {\em regular} if $\det(Q_0)$ has only simple zeros.
\end{Definition}

\begin{Definition}\label{definition-det}
Let $L$ be a line bundle on $\Sigma$ of positive degree and $\delta $ a holomorphic section of $L^2$ with only simple zeros. We say that $(V_0 , Q_0)$ {\it has determinant $(L^* , \delta)$} if there is an isomorphism $\det(V_0) \cong L^*$ under which $\det(Q_0) = \delta$.
\end{Definition}

      \begin{Remark}One should note that   $U${\it-quadratic bundles} $(V_0,Q_0)$, given by a holomorphic vector bundle $V_0$ and $Q_0$ a global section of $Sym^2(V_0^*)\otimes U$ with $U$ a fixed line bundle on $\Sigma$, were considered in \cite{andre} to show that the space of representations into $SO_0(3,2)$ with fixed invariants is connected.  In contrast, in the present paper we are considering only the case of $U\cong\mathcal{O}$ and showing that this suffices to study the space of $SO(p+q,p)$ representations for any $p,q\in \mathbb{N}$.
  \end{Remark}
Given $L$ and $\delta$ as above,   define an associated spectral curve $\pi_C:C\rightarrow \Sigma$ in the total space of $L^2$ by:
\begin{equation}
C = \{ \zeta \; | \; \zeta^2 = \delta \} \subset Tot(L^2). \label{curve}
\end{equation}
 Since $\pi_C : C \to \Sigma$ is a double cover, we have a sheet swapping involution $\sigma_C : C \to C$.  
\begin{Definition}
A $\sigma_C$-equivariant (or simply {\it equivariant}) rank $q$ orthogonal bundle on the double cover $C$ is a triple $(M , Q_M , \tilde{\sigma}_C)$, where $(M , Q_M)$ is a rank $q$ orthogonal bundle on the double cover  $C$, for which there is a lift   $\tilde{\sigma}_C : M \to M$   of $\sigma_C$ to an involution on $M$ which preserves $Q_M$. 
\end{Definition}
Over a fixed point $r$ of $\sigma_C$ on $C$, i.e. a ramification point of $\pi_C$, the involution $\tilde{\sigma}_C$ acts on the fibre $M_r$ of an equivariant rank $q$ orthogonal bundle $(M , Q_M , \tilde{\sigma}_C)$ as an involutive isometry. So there is an orthogonal decomposition $M_r = M_r^+ \oplus M_r^-$ into the $\pm 1$-eigenspaces of $\tilde{\sigma}_C$. Let $q_+ = \dim(M_r^+)$ and $q_- = \dim(M_r^-)$, so $q = q_+ + q_-$. 

\begin{Definition}We say that an equivariant rank $q$ orthogonal bundle $(M , Q_M , \tilde{\sigma}_C)$ has {\em type} $(q_+ , q_-)$ at $r$, for $q_+ , q_-$ obtained as above.  We denote by $\mathcal{M}_C(q_+,q_-)$ the moduli stack of rank $q$ orthogonal bundles $(M , Q_M , \tilde{\sigma}_C)$ on $C$ which have type $(q_+ , q_-)$ over each fixed point. Stability conditions for such bundles will be considered in \autoref{sec:stability}. Note that $\mathcal{M}_C(q_+,q_-)$ is a smooth Artin stack, in fact it is an example of a moduli stack $Bun_{\mathcal{G}}(\Sigma)$ of $\mathcal{G}$-torsors, where $\mathcal{G}$ is a parahoric Bruhat-Tits group scheme on $\Sigma$ (see, \cite[Example (3)]{Hein}).
\end{Definition}

Given an equivariant rank $q$ orthogonal bundle $(M , Q_M , \tilde{\sigma}_C)$ in $\mathcal{M}_C(q_+,q_-)$, we define a rank $q$ quadratic bundle $(V_0 , Q_0)$ as follows:
\begin{itemize}
\item{$V_0$ is defined to be the invariant direct image of $M$, i.e. for each open set $U \subseteq \Sigma$,   set 
$$H^0(U, V_0)= H^0(\pi_C^{-1}(U), M)^{\tilde{\sigma}_C};$$}
\item{$Q_0$ is the restriction of $Q_M$ to $\tilde{\sigma}_C$-invariant sections.}
\end{itemize}

\begin{Proposition}\label{equivM}
There is a bijection between isomorphism classes of rank $q$ orthogonal bundles on $\Sigma$ with fixed determinant $(L^* , \delta)$ and equivariant rank $q$ orthogonal bundles on the spectral curve $C$ associated to $(L , \delta)$, with type $(q-1 , 1)$ over each ramification point. The correspondence is given by taking invariant direct image as described above.
\end{Proposition}
\begin{proof}
Let $(M , Q_M , \tilde{\sigma}_C)$ be an equivariant rank $q$ orthogonal bundle on a spectral curve $C$ given as in Eq.~\eqref{curve} with type $(q-1,1)$ over each ramification point. Let $(V_0 , Q_0)$ be the quadratic bundle on $\Sigma$ given by taking invariant direct image. We must show that $(V_0 , Q_0)$ has determinant $(L^* , \delta)$. For this we just need to show that $\det(Q_0)$ and $\delta$ have the same divisor. Note that the divisor of $\delta$ is exactly the branch locus of $C \to \Sigma$. 

Clearly $Q_0$ is non-degenerate away from the branch points. Consider then a branch point  $x \in \Sigma$  and the corresponding ramification point $r \in C$. Since $\delta$ vanishes to first order at $x$, we can choose a local trivialisation of $L$ and a local coordinate $z$ centered at $x$ such that $\delta(z) = z$.
The tautological section $\zeta$ on $C$ can then be viewed as a local coordinate on $C$ centered at $r$ and satisfying $\zeta^2 = z$. The map $\pi_C$ is given locally by $\pi_C(\zeta) = \zeta^2 = z$ and $\sigma_C$ is given by $\sigma_C(\zeta) = -\zeta$. Let $e_1 , \dots , e_q$ be a local orthonormal frame for $M$. Since $M$ has type $(q-1 , 1)$ at $r$, we can choose $e_1, \dots , e_q$ so that $\tilde{\sigma}_C(e_1) = -e_1$, $\tilde{\sigma}_C(e_j) = e_j$ for $j > 1$. A local frame for the invariant direct image is therefore given by $e'_1 = \zeta e_1,$ and  $e'_2 = e_2 , \dots , e'_q = e_q$. Note that $Q_0( e'_1 , e'_1) = Q_M( \zeta e_1 , \zeta e_1 ) = \zeta^2 = z$ and $Q_0( e'_i , e'_j) = \delta_{ij}$ for $(i,j) \neq (1,1)$. In particular this shows that $\det(Q_0)$ vanishes to first order at $x$, as required.

Conversely, let $(V_0 , Q_0)$ be a quadratic bundle where the divisor of $\det(Q_0)$ is exactly the branch locus of $C \to \Sigma$, and let $x \in \Sigma$ be a branch point. As $\det(Q_0)$ vanishes only to first order at $x$, it follows that $Q_0|_{(V_0)_x}$ has a $1$-dimensional null space $N_x \subseteq (V_0)_x$. Let $\mathcal{N}$ be the sheaf on $C$ consisting of a direct sum of skyscraper sheaves, where for each ramification point $p \in C$, we take the skyscraper sheaf with fibre $N_x^*$, for $x = \pi_C(r)$, located at $r$. 

Then, one may define a vector bundle $M$ on $C$ by the following exact sequence of sheaves
\begin{equation*}
0 \to \mathcal{O}(M) \to \mathcal{O}( \pi_C^*(V_0^*)) \to \mathcal{N} \to 0,
\end{equation*}
where the map $\mathcal{O}( \pi_C^*(V_0^*)) \to \mathcal{N}$ is the direct sum of the maps $(V_0)_x^* \to N_x^*$ dual to the inclusions $N_x \to (V_0)_x$. Around a branch point $x$ we can choose an orthonormal frame $e'_1 , e'_2 , \dots , e'_q$ for $V_0$ where $Q_0(e'_1 , e'_1 ) = z$, and $Q_0(e'_i , e'_j ) = \delta_{ij}$, for $(i,j) \neq (1,1)$. Let $f'_1 , \dots , f'_q$ be the dual frame. Note that $Q_0^{-1}$ defines a singular bilinear form on $V_0^*$ with $Q^{-1}_0( f'_1 , f'_1) = 1/z$, $Q^{-1}_0( f'_i , f'_j) = \delta_{ij}$, for $(i,j) \neq (1,1)$. Then a local frame for $M$ is given by $f_1 = \zeta f'_1,$ and $f_2 = f'_2 , \dots , f_q = f'_q$. 

The restriction  $Q_M$ of $Q^{-1}_0$ to $M$    defines an orthogonal structure on $M$. Then, for  $\tilde{\sigma}_C$ be the restriction to $M$ of the canonical involution on $\pi_C^*(V_0^*)$, we have that $$\tilde{\sigma}_C( f_1 ) = \tilde{\sigma}_C( \zeta f'_1 ) = \sigma_C^*(\zeta) f'_1 = -\zeta f'_1 = -f_1$$ and $\tilde{\sigma}_C( f_j ) = f_j$ for $j > 1$. So $(M , Q_M , \tilde{\sigma}_C )$ has type $(q-1 , 1)$ at each ramification point. Lastly note that these two constructions are inverse to one another so we obtain the desired bijection.
\end{proof}

\begin{Proposition}
  The moduli stack $\mathcal{M}_{C}(q-1,1)$ has dimension $\dim(\mathfrak{so}(q) )(g-1) + (q-1) deg(L)$.\label{Lef}
\end{Proposition}
\begin{proof}
Let $ad(M) \cong \wedge^2 M$ denote the adjoint bundle of $M$. By a standard argument in deformation theory, the dimension of the stack $\mathcal{M}_{C}(q-1,1)$ equals $-\chi^{\tilde{\sigma}_C}$, where $\chi^{\tilde{\sigma}_C}$ is the $\tilde{\sigma}_C$-invariant part of the index:
\begin{equation*}
\chi^{\tilde{\sigma}_C} := \dim( H^0( C , ad(M) )^{\tilde{\sigma}_C} ) - \dim( H^1( C , ad(M) )^{\tilde{\sigma}_C} ).
\end{equation*}
By the Lefschetz index formula (see \cite[Theorem 4.12]{Lef}), we have:
\begin{equation*}
-\chi^{\tilde{\sigma}_C} = \frac{1}{2}\dim( ad(M) )(g_C - 1) - \frac{1}{4} \sum_p tr( \tilde{\sigma}_C : ad(M)_p \to ad(M)_p ),
\end{equation*}
where $g_C$ is the genus of $C$ and the sum is over the ramification points of $C \to \Sigma$. Now since $M$ has type $(q-1 , 1)$ at each ramification point, we see that $$tr( \tilde{\sigma}_C : ad(M)_p \to ad(M)_p ) = \dim(ad(M))_p - 2(q-1) = \dim( \mathfrak{so}(q)) - 2(q-1).$$ Since $\delta$ is a section of $L^2$, letting $d=\deg(L)$, the cover $C \to \Sigma$ has $2d$ ramification points and hence:
\begin{equation*}
\begin{aligned}
-\chi^{\tilde{\sigma}_C} &= \frac{1}{2}\dim( \mathfrak{so}(q) )(g_C - 1) - \frac{d}{2}\left( \dim(\mathfrak{so}(q)) - 2(q-1) \right) \\
&= \frac{1}{2} \dim( \mathfrak{so}(q) ) \left( (g_C-1) - d \right) + (q-1)d.
\end{aligned}
\end{equation*}
By Riemann-Hurwitz, we have $(g_C - 1) - d = 2(g-1)$, so
\begin{equation*}
\chi^{\tilde{\sigma}_C} = \dim( \mathfrak{so}(q) )(g-1) + (q-1)d,
\end{equation*}
which concludes the proof. 
\end{proof}

 \subsection{Quadratic bundles, Higgs bundles and  the auxiliary spectral curve}\label{quad1}
From \autoref{section2} an $SO(p+q,p)$-Higgs bundle $(V , W , \beta)$ satisfying   \autoref{ass1} has a naturally associated rank $q$ quadratic bundle $(V_0,Q_0)$ defined as follows:
\begin{itemize}
\item{$V_0 = Ker( \beta^t) \subset V$ defined as in Eq.~\eqref{W1},}
\item{$Q_0$ is the restriction of $Q_V$ to $V_0$.}
\end{itemize}
 
\begin{Lemma} The symmetric bilinear form $Q_0$ is non-degenerate at points where $a_{p} \neq 0$ and has a $1$-dimensional null space at zeros of $a_{p}$.\end{Lemma}

\begin{proof}
 To see this, recall that we have the short exact sequence 
%\begin{equation*}
$0 \to V_0 \to V \to V_1 \to 0$ 
%\end{equation*}
and the corresponding dual sequence: 
%\begin{equation*}
$0 \to V_1^* \to V \to V_0^* \to 0.$ 
%\end{equation*}
Note also that $V_1^*$ as a subbundle of $V$ can be identified with $V_0^\perp$, the annihilator of $V_0$. Recall that $\beta_+$ viewed as a map $\beta_+ : W \otimes K^{-1} \to V_1$ is given by the composition 
\begin{equation*}\xymatrix{
W \otimes K^{-1} \ar[r]^-{\gamma_+^t} & V_1^* = V_0^\perp \ar[r] & V \ar[r] & V_1 = V/V_0.
}
\end{equation*}
From \autoref{gammaisom} the map $\gamma_+^t$ is everywhere an isomorphism. Moreover, away from the zeros of $a_{p}$ we have that $\beta_+$ is an isomorphism, hence $V_0^\perp \cap V_0 = \{ 0 \}$. Therefore $V_0$ is a non-degenerate subspace of $V$ and hence $Q_0 = Q_V|_{V_0}$ is non-degenerate. If $x$ is a zero of $a_{p}$, then $\beta_+$ has a $1$-dimensional kernel, and hence  $(V_0)^\perp_x \cap (V_0)_x$ is $1$-dimensional. Then,  $Q_0$  at $x$ has a $1$-dimensional null space $N_x = (V_0)^\perp_x \cap (V_0)_x \subseteq (V_0)_x$ as required.\end{proof}

\begin{Lemma}\label{lem:determinants}
For an $O(p+q,p)$-Higgs bundle $(V,W,\beta)$, the condition $\det(V) \cong \det(W)$ is equivalent to requiring $\det(V_0) \cong K^{-p}$.\label{detV0}
\end{Lemma}
\begin{proof}The above argument shows that $\det(V_0^*)^2 \cong K^{2p}$. Using this and the isomorphism $V_1 \cong W \otimes K$, it follows that \[\det(V) \cong \det(V_0)\otimes\det(V_1)% \nonumber\\
\cong \det(V_0)\otimes \det( W \otimes K) %\nonumber\\
\cong \det(V_0) \otimes \det( W ) \otimes K^p.\] %\nonumber
%\end{eqnarray}
% \begin{eqnarray}
%\det(V) &\cong &\det(V_0)\otimes\det(V_1) \nonumber\\
%&\cong &\det(V_0)\otimes \det( W \otimes K) \nonumber\\
%&\cong &\det(V_0) \otimes \det( W ) \otimes K^p.\nonumber
%\end{eqnarray}
Hence $\det(V) \cong \det(W)$ is equivalent to requiring $\det(V_0) \cong K^{-p}$.\end{proof}

From the above analysis, we have the following correspondence between Higgs bundles, quadratic bundles, and orthogonal bundles. 
\begin{Theorem}\label{qtheorem}For each $SO(p+q,p)$-Higgs bundle $(V , W , \beta)$ there is a rank $q$ quadratic bundle $( V_0 , Q_0) = (Ker(\beta^t) , Q_V|_{Ker(\beta^t)})$ which has determinant $( K^{-p} , a_{p})$. Letting $\pi_C : C \to \Sigma$ denote the corresponding  curve  $\zeta^2 = a_{p}$, the quadratic bundle $(V_0 , Q_0 )$ is given as the invariant direct image of an equivariant orthogonal bundle $(M , Q_M , \tilde{\sigma}_C ) \in  \mathcal{M}_C(q-1,1)$.
\end{Theorem}
\begin{Remark}
Recall that to define an $SO(p+q,p)$-Higgs bundle it is not enough just to have an isomorphism $\det(V) \cong \det(W)$. We must actually fix a choice of isomorphism of $\mathbb{Z}_2$-line bundles. So far we have not identified what this choice corresponds to in terms of the equivariant bundle $(M , Q_M , \tilde{\sigma}_C)$, but we will return to this point later.
\end{Remark}

  \section{Abelian and Non-Abelian data for\\ orthogonal Higgs bundles}\label{section4}
  
  We have described in previous sections how to associate to an $SO(p+q,p)$-Higgs bundle $(V,W, Q_V , Q_W , \beta)$, a $K^2$-twisted $GL(p,\mathbb{R})$-Higgs bundle $(W , Q_W , \beta_F)$ via a Cayley type correspondence, as well as a quadratic bundle $(V_0 , Q_0)$, as part of the Langlands type correspondence. To complete the Langlands correspondence, we need to understand the extension data required to reconstruct $(V , Q_V)$ from $(W , Q_W, \beta_F)$ and $(V_0 , Q_0)$. Finally, we also consider the relation between stability conditions for the $SO(p+q,p)$-Higgs bundle and the quadratic bundle.     
  
  \subsection{The extension data} We first look into how a rank $p+q$ vector bundle $V$ and its orthogonal structure $Q_V$ can be recovered from a $K^2$-twisted $GL(p , \mathbb{R})$-bundle $(W , Q_W , \beta_F)$ and a quadratic bundle $(V_0 , Q_0)$ as in \autoref{qtheorem}. In particular, recall from \autoref{uni} that $V_{0}={\rm ker}(\beta^t)$ and $V_1=V/V_0$, so we have the short exact sequence:
%\begin{equation*}
$0 \to V_0 \to V \to V_1 \to 0$. 
%\end{equation*}
Let $D \subset \Sigma$ denote the zero divisor of $a_{p}$. Then since $V_0$ is non-degenerate on $\Sigma \setminus D$, we obtain an orthogonal splitting:
\begin{equation}\label{equ:splitv}
 V \cong V_0^\perp \oplus V_0 
\cong V_1^* \oplus V_0 
 \cong (W \otimes K^{-1}) \oplus V_0.
 \end{equation}

\begin{Lemma}\label{lem1}
Let $(V,W,\beta)$ be an $SO(p+q,p)$-bundle. Then with respect to the splitting \eqref{equ:splitv} on $\Sigma \setminus D$, the orthogonal structure on $V$ is given by 
\begin{equation}\label{equ:qv}
Q_V( (a , b) , (c , d) ) = Q_W( \beta_F a , c) + Q_0(b , d).
\end{equation}
Conversely, given  $(V_0 , Q_0)$ in and a  $K^2$-twisted $GL(p , \mathbb{R})$ Higgs bundle $(W , Q_W , \beta_F)$, Equation  \eqref{equ:qv} defines an orthogonal structure on $(W \otimes K^{-1}) \oplus V_0$ over $\Sigma \setminus D$.
\end{Lemma}
\begin{proof} Suppose we are given an $SO(p+q,p)$-Higgs bundle $(V,W,\beta)$. The orthogonal structure $Q_V$ on $V$ is equivalent to giving a symmetric map $Q_V : V \to V^*$, i.e. a map $Q_V : (V_1^* \oplus V_0) \to (V_1 \oplus V_0^*)$. Since the direct sums are orthogonal, this is equivalent to giving maps $V_0 \to V_0^*$ and $V_1^* \to V_1$. Clearly the map $V_0 \to V_0^*$ is given by $Q_0$. From \eqref{phif} one finds that the map $V_1^* \cong W \otimes K^{-1} \to V_1 \cong W \otimes K$ in \eqref{D4} is given by $\beta_F : W \otimes K^{-1} \to W \otimes K$. In other words, with respect to the splitting $V \cong (W \otimes K^{-1}) \oplus V_0$, we have $Q_V( (a , b) , (c , d) ) = Q_W( \beta_F a , c) + Q_0(b , d)$. 

Conversely, suppose we are given $(V_0,Q_0)$ and $(W,Q_W , \beta_F)$. Since $\det(\beta_F) = (-1)^p a_p$, we have that $\beta_F$ is an isomorphism on $\Sigma \setminus D$. Moreover, $Q_0$ is non-degenerate on $\Sigma \setminus D$ and thus $Q_V$ is non-degenerate. Furthermore,  since $Q_W(\beta_F a , c) = Q_W( a , \beta_F c)$, the form $Q_V$ is symmetric.
\end{proof}

\begin{Lemma} \label{lem2} 
Let $(V,W,\beta)$ be an $SO(p+q,p)$-bundle. Then with respect to the orthogonal splitting $V \cong (W \otimes K^{-1}) \oplus V_0$ as in \eqref{equ:splitv}, we have that $\beta$ and $\gamma$ are given by
\begin{equation}\label{equ:betaandgamma}
%\begin{aligned}
\beta(u) = (u , 0), \quad \gamma(v,w)= \beta_F(v).
%\end{aligned}
\end{equation}
Conversely, given a rank $q$ orthogonal bundle $(V_0 , Q_0)$ and a  $K^2$-twisted $GL(p , \mathbb{R})$ Higgs bundle $(W , Q_W , \beta_F)$, we obtain an induced  $SO(p+q,p)$-Higgs field on $V\oplus W$ on $\Sigma \setminus D$ where $(V , Q_V)$ is as in \autoref{lem1} and $\beta, \gamma$ are as in \eqref{equ:betaandgamma}.
\end{Lemma}
\begin{proof}
From an $SO(p+q,p)$-Higgs bundle $(V,W,\beta)$,  one may obtain $\beta$ as the composition:
\begin{equation*}
W \cong V_1^* \otimes K \to V \otimes K.
\end{equation*}
Under the splitting $V = V_1^* \oplus V_0$ we have that $V \otimes K = (V_1^* \otimes K) \oplus (V_0 \otimes K) \cong W \oplus (V_0 \otimes K),$ and $\beta : W \to V \otimes K$ is the inclusion of the first factor. Similarly $\gamma : V \to W \otimes K$ is the map on $V = V_1^* \oplus V_0 \cong (W \otimes K^{-1}) \oplus V_0$ given by:
\begin{equation*}
V \cong (W \otimes K^{-1}) \oplus V_0 \to W \otimes K^{-1} \buildrel \beta_F \over \longrightarrow W \otimes K,
\end{equation*}
where the first arrow is the projection to the first factor, and hence $\beta, \gamma$ are given as in \eqref{equ:betaandgamma}.

Conversely, from $(V_0,Q_0)$ and $(W,Q_W , \beta_F)$ we have that \[  Q_V( \beta(u) , (v,w) ) = Q_V( (u,0) , (v,w) ) 
= Q_W( \beta_F u , v ) 
= Q_W( u , \beta_F v ) 
= Q_W( u , \gamma(v,w)).
\]
%\begin{equation*}
%\begin{aligned}
%Q_V( \beta(u) , (v,w) ) &= Q_V( (u,0) , (v,w) ) \\
%&= Q_W( \beta_F u , v ) \\
%&= Q_W( u , \beta_F v ) \\
%&= Q_W( u , \gamma(v,w)).
%\end{aligned}
%\end{equation*}
So we have shown that $\beta^t = \gamma$, as required for an $SO(p+q,p)$-Higgs bundle. The condition $det(V) = det(W)$ follows from $det(V_0) = K^{-p}$, as in the proof of \autoref{lem:determinants}.
\end{proof}

In what follows we extend the constructions of   \autoref{lem1}-\autoref{lem2} over the divisor $D$. To do this, we will define $V$ as an extension of $V_0^*$ by $V_1^*$, hence given by an extension class in $H^1( \Sigma , Hom( V_0^* , V_1^* ) )$. As this extension class needs to be trivial on $\Sigma \setminus D$, it  allows us to define $Q_V , \beta , \gamma$ as in   \autoref{lem1}-\autoref{lem2}  on the complement of $D$. We then need to check that $Q_V , \beta , \gamma$ extend over $D$. By continuity, if such an extension exists, it is uniquely determined. Moreover, we have that $deg(V) = 0$, so if $Q_V$ extends, the extension will automatically be non-degenerate.

We consider $Q_0$   as a map $Q_0 : V_0 \to V_0^*$ which   induces   $Q_0 : Hom( V_0^* , V_1^* ) \to Hom( V_0 , V_1^*)$. Then, letting   $N$  be the vector bundle over the finite set $D$ whose fibre over $x \in D$ is given by $N_x = (V_0)_x \cap (V_0^\perp)_x \cong (V_0)_x \cap (V_1^*)_x$,   there is an exact sequence at each  $x \in D$ :
\begin{equation*}
0 \to N_x \to (V_0)_x \buildrel Q_0 \over \longrightarrow (V_0^*)_x \to N_x^* \to 0.
\end{equation*}
From the above, there is a short exact sequence of sheaves:
\begin{equation*}
0 \to \mathcal{O}_\Sigma( Hom(V_0^* , V_1^*)) \buildrel Q_0 \over \longrightarrow \mathcal{O}_\Sigma( Hom( V_0 , V_1^* ) ) \to \mathcal{O}_D( Hom(N , V_1^* )  ) \to 0,
\end{equation*}
which determines a long exact sequence:
\begin{equation*}
\cdots \to H^0(D , Hom( N , V_0^\perp ) ) \buildrel \delta \over \longrightarrow H^1( \Sigma , Hom( V_0^* , V_1^*)) \buildrel Q_0 \over \longrightarrow H^1( \Sigma , Hom( V_0 , V_1^*)) \to \cdots
\end{equation*}
where we use the inner product on $V$ make the identification $V_1^* \cong V_0^\perp$. Let $i : N \to (V_0)^\perp |_D$ be given by the inclusion $N_x = (V_0)_x \cap (V_0^\perp)_x \to (V_0^\perp)_x$. Then $i \in H^0( D , Hom( N , V_0^\perp ) )$ and $\delta(i) \in H^1( \Sigma , Hom( V_0^* , V_1^* ) )$ is an extension class.

\begin{Proposition}\label{propExt}
For any $(V , W , \beta)$, the class $\delta(i)$ is the extension class defining $V$.\label{EC}
\end{Proposition}
\begin{proof}
Consider an open cover of  $\Sigma$ by two open sets $U$ and $\Sigma \setminus D$, where $U$ is a disjoint union of small discs around each point of $D$, and consider \v{C}ech cocycles  with respect to this cover. On $\Sigma \setminus D$, we have the orthogonal splitting
\begin{equation}\label{equ:ortsplit}
V \cong V_1^* \oplus V_0 \cong V_1^* \oplus V_0^*,
\end{equation}
where $Q_0$ is used to identify $V_0$ and $V_0^*$ on the complement of $D$. On $U$, we choose a local splitting $V \cong V_1^* \oplus V_0^*$ of the short exact sequence 
\begin{equation*}
0 \to V_1^* \to V \to V_0^* \to 0.
\end{equation*}
With respect to this splitting, the inclusion $V_0 \to V$ in \eqref{equ:ortsplit} has the form 
\begin{eqnarray}
V_0&\rightarrow& V_1^* \oplus V_0^*\\
s&\mapsto& (u(s) , Q_0(s)),
\end{eqnarray}
 where $u$ is some locally defined holomorphic map $u : U \to Hom( V_0 , V_1^* )$ such that $u|_N = i$ on $D$. Therefore, the local trivialisations on $\Sigma \setminus D$ and $U$ are related on $(\Sigma \setminus D) \cap U$ as follows:
\begin{equation*}
(v  , s ) \mapsto (v + u(s) , Q_0(s) ),
\end{equation*}
where $(v ,s) \in V_1^* \oplus V_0$ is a local section defined with respect to the orthogonal splitting and $(v + u(s) , Q_0(s) ) \in V_1^* \oplus V_0^*$ is the corresponding section in the local splitting given on $U$. So the extension class in $H^1(\Sigma , Hom( V_0^* , V_1^*))$ defining $V$ is represented by
\begin{equation*}
\left( w \mapsto u( Q_0^{-1}(w) ) \right) \in H^0( (\Sigma \setminus D) \cap U , Hom( V_0^* , V_1^* ) ).
\end{equation*}
Since $u|_N = i$ on $D$, it is straightforward to see that this extension class is exactly $\delta(i)$.
\end{proof}
From \autoref{EC}   the data required to define the extension $V$ is a map 
\begin{eqnarray}
i : N \to V_1^* \cong W \otimes K^{-1},\label{i}
\end{eqnarray}
defined over $D$. Given $(W , Q_W, \beta_F)$ and $(V_0 , Q_0)$, we would like to determine for which such $i$, the triple $(Q_V , \beta , \gamma)$ defined as in \autoref{lem2} extends over $D$. In order to simplify the subsequent computations, we give an alternative interpretation of the vector bundle $V$.

\begin{Proposition}\label{prop:merogamma}
Let $i \in H^0( D , Hom(N , V_1^*)$ and let $V$ be the extension of $V_0^*$ by $V_1^*$ determined by $\delta(i)$. Then $\mathcal{O}(V)$ is isomorphic to the sheaf of meromorphic sections of $V_1^* \oplus V_0$ admitting first order poles over $D$ whose residue over $x \in D$ lies in the subspace
\begin{equation}\label{equ:gammax}
\Gamma_x = \{ ( -i(w) , w ) \; | \; w \in N_x \} \subseteq (V_1^*)_x \oplus (V_0)_x.
\end{equation}
The map $\mathcal{O}(V_1^*) \to \mathcal{O}(V)$ is the   inclusion and the map $\mathcal{O}(V) \to \mathcal{O}(V_0^*)$ sends a meromorphic section $(v,s)$ of $V_1^* \oplus V_0$ to $Q_0(s) \in \mathcal{O}(V_0^*)$.
\end{Proposition}
\begin{proof}
We have just seen that $V \cong V_1^* \oplus V_0$ on $\Sigma \setminus D$, that $V \cong V_1^* \oplus V_0^*$ on $U$, with $U$  an open neighbourhood of the $D$, and the transition maps $V_1^* \oplus V_0 \to V_1^* \oplus V_0^*$ have the form
\[
(v  , s ) \mapsto (v + u(s) , Q_0(s) ),
\]
where $u$ is some locally defined holomorphic map $u : U \to Hom( V_0 , V_1^* )$ such that $u|_N = i$ on $D$. The inverse map $V_1^* \oplus V_0^* \to V_1^* \oplus V_0$ is thus 
$
( v, s) \mapsto (v - u(Q_0^{-1}(s)) , Q_0^{-1}(s) ).
$ 
This shows that holomorphic sections of $V$ can be identified with meromorphic sections of $V_1^* \oplus V_0$ with first order poles over $D$, whose residue over $x \in D$ lies in the subspace $\Gamma_x$ given by \eqref{equ:gammax}. The rest of the proposition follows immediately.
\end{proof}

\begin{Proposition}\label{mapi}
Given a $K^2$-twisted $GL(p , \mathbb{R})$-bundle $(W , Q_W , \beta_F)$ and a quadratic bundle $(V_0 , Q_0)$, the induced maps $\beta , \gamma$ extend over $D$ if and only if
\begin{eqnarray}
i ( Ker(Q_0) ) \subseteq Ker(\beta_F).\label{ii}
\end{eqnarray}
 \end{Proposition}
 \begin{proof}
Since $\beta = \gamma^t$, it suffices to only check that $\gamma$ extends if and only if $i$ satisfies \eqref{ii}. In light of \autoref{prop:merogamma}, we just need to check that if $(v,w)$ is a meromorphic section of $V_1^* \oplus V_0$ whose poles over each $x \in D$ lie in $\Gamma_x$, then $\gamma(v,w) = \beta_F(v)$ is holomorphic (has no poles). Clearly this holds for all $(v,w) \in \mathcal{O}(V)$ if and only if $(\beta_F)_x ( i(N_x) ) = \{ 0 \}$ for each $x \in D$.
\end{proof}

From the above propositions,  the map $i : N \to V_1^*$ is required to send $N$ to the kernel of $\beta_F : V_1^* \cong W \otimes K^{-1} \to W \otimes K \cong V_1$. Over $D$, let $J$ denote the kernel of $\beta_F$. Then for each $x \in D$, $J_x$ is a $1$-dimensional space. We require that $i$ has the form
\begin{equation}
i : N = Ker(Q_0) \to J= Ker(\beta_F)  \subseteq V_1^*.\label{kernelJ}
\end{equation}
Since $N$ and $J$ are $1$-dimensional, there is for each $x \in D$ a $1$-dimensional space of such maps.

\begin{Proposition}\label{condI}\label{p8}Given $i : N \to J$, the bilinear form $Q_V$ extends holomorphically if  and only if  
\begin{equation}\label{equ:isometry}
i^*\left( \frac{Q_W}{a_{p-1}} \right) = \left( \frac{ Q_0 }{a_{p}} \right).
\end{equation}
\end{Proposition}

\begin{Remark}
 \eqref{equ:isometry} is to be understood as follows: the map $i_x : N_x \to J_x$ is equivalent to a map $i_x : N_x \otimes K^p_x \to J_x \otimes K^p_x$. We have that $Q_W$ is a non-degenerate bilinear form on $W \cong V_1^* \otimes K$ and $J_x \subseteq (V_1^*)_x \cong W_x \otimes K^{-1}_x$. Then since $a_{p-1}(x) \neq 0$, we have that $\frac{Q_W}{a_{p-1}}$ is a non-degenerate bilinear form on $W_x \otimes K_x^{p-1} \cong (V_1^*)_x \otimes K^p_x$ and restricts to a non-degenerate bilinear form on $J_x \otimes K^p_x$. Similarly, since $det(Q_0)$ and $a_p$ vanish to first order at $x$, one sees that $\frac{Q_0}{a_p}$ is a well-defined non-degenerate bilinear form on $N_x \otimes K_x^p$. \eqref{equ:isometry} says that $i_x : N_x \otimes K_x^p \to J_x \otimes K^p_x$ is an isometry of $1$-dimensional orthogonal spaces.
\end{Remark}

\begin{Remark}
Alternatively, $\frac{Q_W}{a_{p-1}} \oplus \frac{Q_0}{a_p}$ defines an orthogonal structure on the $2$-dimensional space $(J_x \oplus N_x) \otimes K_x^p$ and thus a conformal structure on $(J_x \oplus N_x)$. Condition \eqref{equ:isometry} is equivalent to saying that $\Gamma_x \subset J_x \oplus N_x$ given by \eqref{equ:gammax} is an isotropic subspace. 
\end{Remark}

\begin{proof} {\it (of \autoref{p8})}
By \autoref{prop:merogamma}, we just need to show that if $(v,w)$ is a meromorphic section of $V_1^* \oplus V_0$ whose poles over each $x \in D$ lie in $\Gamma_x$, then 
\begin{equation}\label{equ:qv}
Q_V ( (v,w) , (v,w) ) = Q_W( \beta_F v , v) + Q_0(w,w)
\end{equation}
is holomorphic. For this, consider the local behaviour around a given $x \in D$. Let $z$ be a local holomorphic coordinate centred at $x$ and use $dz$ to trivialise the canonical bundle. Then we can write $a_p = a_p(z)(dz)^{2p}$, and $a_{p-1} = a_{p-1}(z)(dz)^{2p-2}$, where by \autoref{ass1}, we have $a_p(0) = 0$, and  $a'_p(0) , a_{p-1}(0) \neq 0$. This means that $\lambda = 0$ is an eigenvalue of $\beta_F$ with multiplicity $1$ at $z=0$. Thus we can find a local orthonormal frame $e_1 , \dots , e_p$ for $W$ such that $\beta_F(z) e_1 = \lambda(z)e_1 (dz)^2$, where $\lambda(z)$ is a locally defined holomorphic function with $\lambda(0) = 0$. In particular, $Ker(\beta_F)|_{z=0}$ is spanned by $e_1|_{z=0}$. With respect to the given frame, $\beta_F(z)$ is a matrix of functions times $(dz)^2$. Differentiating $\beta_F(z) e_1 = \lambda(z) e_1 (dz)^2$, we get $\beta'_F(0)e_1 = \lambda'(0)e_1(dz)^2$. Let $\hat{w}$ be a local non-vanishing holomorphic section of $V_0$ such that $\hat{w}|_{z=0}$ spans $N_x$. Choosing $\hat{w}$ so that $Q_0( \hat{w} , \hat{w} ) = a_p(z)$, we have that $i_x$ is given by $i( \hat{w}|_{z=0}) = u e_1|_{z=0}$ for some $u \in \mathbb{C}$. Then $(v,w)$ can be written as
\[
v(z) = v'(z) - \frac{cu}{z}e_1, \quad \quad w(z) = w'(z) + \frac{c}{z}\hat{w},
\]
where $v' , w' $ are holomorphic and $c$ is a constant. Putting these into \eqref{equ:qv}, since $(\beta_F e_1 )|_{z=0} = 0$, we get that $Q_V( (v,w) , (v,w) ) $ is 
\begin{align}
%& Q_V( (v,w) , (v,w) ) = \\
& \quad \quad \frac{1}{z} \left( c^2 a'_p(0) + c^2 u^2 Q_W( \beta'_F(0) e_1(0) , e_1(0) ) - cu Q_W( \beta_F(0)v'(0) , e_1(0)) \right) + \cdots\label{QV}
\end{align}
where $\cdots$ denotes terms which are holomorphic. But  since   $$Q_W( \beta_F(0) v'(0) , e_1(0) ) = Q_W( v'(0) , \beta_F(0) e_1(0) ) = 0,$$ we have that \eqref{QV} is holomorphic for all $(v,w)$ if and only if
\[
Q_W( \beta'_F(0) i( \hat{w}(0) ) , i( \hat{w}(0) ) ) = -a'_p(0).
\]
But $\beta'_F(0) i(\hat{w}(0)) = u \beta'_F(0) e_1(0) = u \lambda'(0)e_1(0) (dz)^2 = \lambda'(0) i(\hat{w}(0)) (dz)^2$, so this simplifies to
\[
\lambda'(0) Q_W( i( \hat{w}(0) ) , i( \hat{w}(0) ) ) (dz)^2 = -a'_p(0).
\]
Now if we use the fact that the characteristic polynomial $p(y,z) = y^p + \dots + a_{p-1}(z)y + a_p(z)$ of $\beta_F(z)$ factors (locally) as $p(y,z) = ( y - \lambda_1(z) )( \hat{p}(y,z) )$ for some $\hat{p}(y,z)$, we find that $\lambda'(0) = -\frac{a'_p(0)}{a_{p-1}(0)}$. Therefore
\[
\frac{Q_W}{a_{p-1}(0) (dz)^{2p-2} }( i(\hat{w}(0)) , i(\hat{w}(0)) ) = \frac{1}{(dz)^{2p}} = \left. \frac{Q_0}{a_p(z)(dz)^{2p}}( \hat{w}(z) , \hat{w}(z) ) \right|_{z=0},
\]
which is exactly \eqref{equ:isometry}.
\end{proof}

 From \autoref{p8}, we see that given a $K^2$-twisted $GL(p , \mathbb{R})$-bundle  $(W , Q_W , \beta_F)$ and a quadratic bundle $(V_0 , Q_0)$, for each point $x \in D$ there are exactly two choices of the map $i_x$ for which the induced triple $(Q_V,\beta , \gamma)$ extends over $x$.
 
\begin{Proposition}
Let $(W , Q_W , \beta_F)$ be the $K^2$-twisted $GL(p , \mathbb{R})$-Higgs bundle corresponding to an orthogonal line bundle $L \in Jac(\overline{S})[2]$ and let $(V_0 , Q_0)$ be the quadratic bundle given by the invariant direct image of an equivariant orthogonal bundle $(M , Q_M , \tilde{\sigma}_C)$. The extension data needed to obtain an $SO(p+q,p)$-Higgs bundle from $L$ and $(M , Q_M , \tilde{\sigma}_C)$ is an isometry $\tau_x : M_r^- \to L_{r'}$ over each  $x \in D$, where $r$ is the ramification point of $\pi_C$ lying over $x$, $r'$ is the ramification point of $\rho : S \to \overline{S}$ over $x$ and $M_r^-$ is the $-1$-eigenspace of $\tilde{\sigma}_C$ on $M_r$.\end{Proposition}
\begin{proof}
Given $e \in M_r^-$, choose a local section $\tilde{e}$ of $M$ such that $\tilde{e}|_r = e$ and $\tilde{\sigma}_C(\tilde{e}) = -\tilde{e}$. Recall that $\zeta$ is the tautological section of $\pi_C^* K^p$ on $C$. It follows that $\zeta \tilde{e}$ is a $\pi_C^* K^p$-valued $\tilde{\sigma}_C$-invariant section of $M$ and thus defines a local $K^p$-valued section of $V_0$. The restriction $\zeta \tilde{e}|_x \in (V_0)_x \otimes K^p_x$ is easily seen to lie in $N_x \otimes K_x^p$ and is independent of the choice of extension $\tilde{e}$. In this way, we have defined a canonical map 
$
j_x : M_r^- \to N_x \otimes K_x^p.
$ Since $\zeta^2 = a_p$, we find that:
\begin{equation*}
j^*_x \left( \frac{Q_0}{a_{p}} \right) = Q_M |_{M_r^-}.
\end{equation*}
Hence, given
$i_x : N_x   \to J_x 
$
as in Eq.~\eqref{kernelJ} and 
 letting $\kappa_x = i_x \circ j_x : M_r^- \to J_x \otimes K_x^p$,   the condition on $i_x$ from  \autoref{condI} now becomes:
\begin{equation}
\kappa_x^* \left( \frac{Q_W}{a_{p-1}} \right) = Q_M|_{M_r^-}. \label{kappa1}
\end{equation}
Let $L \in Jac(S)[2]$ be the line bundle on $\overline{S}$ such that $W = \overline{\pi}_*(L \otimes \overline{\pi}^* K^{(p-1)})$ (see \cite{BNR} for details). Note that by construction $K_{\overline{S}} \cong \overline{\pi}^* K^{2p-1}$ and hence $K_{\overline{S}}\otimes \overline{\pi}^* K^{-1} \cong \overline{\pi}^* K^{2p-2}$. The orthogonal structure on $L$ induces the orthogonal structure on $W$ by relative duality.

We have that $J_x  \subseteq (V_1)_x^* \cong W_x \otimes K_x^{-1}$. Thus $J_x \otimes K_x$ can be identified with the kernel of $(\beta_F)_x : W_x \to W_x \otimes K_x^2$. Since $\beta_F$ is obtained by pushing down the tautological section $\xi$, it follows that $J_x \otimes K_x$ is canonically isomorphic to $L_{r'} \otimes K_x^{(p-1)}$. Therefore $J_x \otimes K_x^p$ is canonically isomorphic to $L_{r'} \otimes K_x^{2p-2}$ and $i_x$ corresponds to a map $\kappa_x : M_r^- \to L_{r'} \otimes K_x^{2p-2}$. 

Recall from Eq.~\eqref{curvep} in \autoref{orto} that the $p$-fold cover  $\overline \pi:\overline{S}\rightarrow \Sigma$ is given by the equation
\begin{equation*}
\xi^p + a_1 \xi^{p-1} + \dots + a_{p-1} \xi + a_{p} = 0.
\end{equation*}
The derivative $d \overline{\pi} : T_{\overline{S}} \to \overline{\pi}^* T_\Sigma$ defines a section of $K_{\overline{S}} \otimes \overline{\pi}^* K^{-1}$. Under the isomorphism $K_{\overline{S}} \otimes \overline{\pi}^* K^{-1} \cong \overline{\pi}^* K^{2p-2}$, we have that $d\overline{\pi}$ is given by $p \xi^{p-1} + (p-1)a_1 \xi^{p-2} + \dots + a_{p-1}$. So at   
  the ramification point $r' \in \overline{S}$, we have $d\overline{\pi} = a_{p-1}(x) \neq 0$, by \autoref{ass1}. Therefore, relative duality gives 
\begin{equation*}
Q_W|_{J_x \otimes K_x} = \frac{ Q_{L}}{d\overline{\pi}} = \frac{ Q_{L}}{a_{p-1}(x)},
\end{equation*}
where $ Q_{L}$ is the orthogonal structure on $L$. Putting it all together, the data we need are a maps $\iota_x : M_r^- \to L_{r'} \otimes K_x^{2p-2}$ such that Eq.~\eqref{kappa1} now becomes:
\begin{equation}
\iota^*_x \left( \frac{ Q_{L}}{a_{p-1}^2(x)} \right) = Q_M|_{M_r^-}. \label{kappa2}
\end{equation}

Consider now the map $$\tau_x : M_r^- \to L_{r'}$$ defined  by setting $\tau_x = \iota_x \otimes a_{p-1}^{-1}(x)$. Then the above condition in Eq.~\eqref{kappa2} is simply that $\tau_x$ is an isometry. So finally we have identified the extension data: for each point $x \in D$, let $r \in C$ and $r' \in \overline{S}$ be the corresponding ramification points. Then we need an isometry $\tau_x : M_r^- \to L_{r'}$.
\end{proof}

\begin{Remark}
Replacing $i$ by $-i$  (changing the sign of $i$ at every point of $D$) we get an isomorphic extension and an isomorphic Higgs bundle (the isomorphism sending $i$ to $-i$ is given by acting as $-1$ on $W$ and $V_1$ and as $1$ on $V_0$. This has determinant $1$, since ${\rm rank}(V_1) = {\rm rank}(W)$).
\end{Remark}

\subsection{Determinant bundles}
  
To complete the correspondence between $SO(p+q,p)$-Higgs bundles and triples $(W, Q_W , \beta_F) , (V_0 , Q_0) , \{ i_x \}_{x \in D}$, we still need to identify the data giving the isomorphism $\det(V) \cong \det(W)$ of $\mathbb{Z}_2$-line bundles for the induced $SO(p+q,p)$-Higgs bundle. Since $\det(V), \det(W)$ are local systems, it is enough to give such an isomorphism on the complement of $D$ and check this isomorphism extends. Using the fact that on a compact Riemann surface, unitary flat line bundles correspond to degree $0$ holomorphic line bundles, we see that such an isomorphism will extend if and only if $Hom( \det(V) , \det(W) ) \cong \mathcal{O}$, i.e. $\det(V) \cong \det(W)$ as holomorphic line bundles. We have already seen in \autoref{detV0} that this holds if and only if $\det(V_0) \cong K^{-p}$. So in what follows we will assume $\det(V_0) \cong K^{-p}$ and look for an isomorphism $\det(V) \cong \det(W)$ of $\mathbb{Z}_2$-line bundles away from $D$.

\begin{Proposition}
The choice of isomorphism $\det(V) \cong \det(W)$ is given by a choice of orientation for $M$, the equivariant orthogonal bundle corresponding to $(V_0 , Q_0)$.
\end{Proposition}
\begin{proof}
On the complement of $D$ we have $V \cong V_1^* \oplus V_0 \cong (W \otimes K^{-1}) \oplus V_0$ and
\begin{equation*}
Q_V( (a,b) , (c,d) ) = Q_W( \beta_F a , c ) + Q_0(b,d).
\end{equation*}

Let $vol_W$ be a local unit volume form for $W$. Away from $D$ we have $a_{p} \neq 0$ and we can locally choose a square root $\zeta$, a section of $K^p$ with $\zeta^2 = a_{p}$. Then since $\det(\beta_F) = (-1)^p a_p$, we have that $\zeta^{-1} vol_W$ is a local unit volume form for $(W \otimes K^{-1})$ with respect to the bilinear form sending $\omega_1,\omega_2 \mapsto Q_W( \beta_F \omega_1 , \omega_2)$. 
Since $\det(V_0) \cong K^{-p}$ we may locally away from $D$ choose an isomorphism $\psi : K^{p} \to \det(V_0^*)$, such that the induced map $\psi^{2} : K^{2p} \to \det(V_0^*)^2$ satisfies $\psi^{2}( a_{p} ) = \det(Q_0)$. This implies that $\psi^{-1}( \zeta^{-1} ) \in \mathcal{O}( \det(V_0) )$ is a local unit volume form for $V_0$. Consider then $\varphi : \det(W) \to \det(V)$ defined by:
\begin{equation*}
\varphi( vol_W ) = (\zeta^{-1} vol_W ) \wedge \psi( \zeta^{-1} ).
\end{equation*}

Note that $\varphi$ does not depend on the choice of (locally defined) square root $\zeta$ of $a_{p}$. Thus $\varphi$ is globally defined on the complement of $D$ and depends only on $\psi$. In this manner, the choice of $\varphi$ is equivalent to the choice of $\psi$.
Hence, the determinant data is given by the choice of an isomorphism $\psi : K^p \to \det(V_0^*)$ such that $\psi^2( a_{p} ) = \det(Q_0)$

Consider  the equivariant orthogonal bundle $(M , Q_M , \tilde{\sigma}_C )$   corresponding to $(V_0 , Q_0)$. Then the condition $\det(V_0^*) \cong K^p$ is equivalent to requiring that $\det(M) = \mathcal{O}$ (see \cite[Section 4]{BNR}). Let $vol_M$ be a global unit volume form on $M$ (there are precisely two choices for $vol_M$). We also have $\tilde{\sigma}_C( vol_M) = -vol_M$, which can be seen by considering $vol_M$ around a ramification point. As before, let $\zeta$ be the tautological section of $\pi_C^*(K^p)$, which satisfies $\zeta^2 = a_{p}$. Then $vol_M \zeta$ is an invariant section of $\pi^*( \det(V_0) \otimes K^p)$ and descends to an isomorphism $\psi : K^p \to \det(V_0^*)$ with the desired property. Thus, the choice of isomorphism $\det(V) \cong \det(W)$ corresponds naturally to a choice of orientation for $M$.
\end{proof}

%%%%%%%%%%%%%%%%%%%%%%%%%%%%%%%%%%%%%%%%%%%
      \subsection{Stability}\label{sec:stability}

To complete the Langlands type correspondence, it remains to relate stability of the $SO(p+q,p)$-Higgs bundle to a stability condition on the quadratic bundle $(V_0 , Q_0)$, or equivalently to a stability condition on the corresponding equivariant bundle $(M , Q_M , \tilde{\sigma}_C )$. An $SO(p+q,p)-$Higgs bundle $(V,W, Q_V, Q_W , \beta)$ is semistable if and only if for all pairs of isotropic subbundles $V' \subset V$, $W' \subset W$ with $\beta( W') \subseteq V' \otimes K$ and $\gamma(V') \subseteq W' \otimes K$, we have $\deg(V') + \deg(W') \le 0$ \cite{aparicio}. Notice that the conditions on $(V' , W')$ are equivalent to saying $V' \oplus W'$ is an isotropic in $E = V \oplus W$ which is invariant under $\Phi$.
 
 \begin{Lemma}
For an $SO(p+q,p)-$Higgs bundle $(V,W, Q_V, Q_W , \beta)$ in a regular fibre of the Hitchin fibration, if $V',W'$ is a pair of isotropics with $\beta( W') \subseteq V' \otimes K$ and $\gamma(V') \subseteq W' \otimes K$, then $W' = 0$ and $V'\subseteq V_0$. 
 \end{Lemma}
\begin{proof} Suppose that the spectral curve $\pi:S\rightarrow \Sigma$ from \autoref{orto} is smooth. Then its defining polynomial 
$$\eta^{2p} + a_1 \eta^{2p-2} + \dots + a_{p}$$ is irreducible and therefore the only invariant subbundles of $V_1 \oplus W$ are $V_1 \oplus W$ and $\{ 0 \}$. If there were isotropic bundles $V' , W'$ with $\beta(W') \subseteq V' \otimes K$ and $\gamma(V') \subseteq W' \otimes K$, then the image of $V' \oplus W'$ under $V \oplus W \to V_1 \oplus W$ would need to be either $V_1 \oplus W$ or $\{ 0 \}$. However, since $V' \oplus W'$ is isotropic,   the image has to be $\{ 0 \}$ as it can not contain $W$. This implies that $W' = \{ 0 \}$ and $V' \subseteq V_0$.
\end{proof}
We have therefore shown that semistability of the $SO(p+q,p)-$Higgs bundle \linebreak$(V,W, Q_V, Q_W , \beta)$ reduces to:
\begin{center}{  \it For all isotropic subbundles $V' \subset V_0$, we have $\deg(V') \le 0$.
}\end{center}

Let $(M , Q_M , \tilde{\sigma}_C )$ be the corresponding equivariant orthogonal bundle. Then from above analysis one can understand semistability   as follows.

\begin{Proposition} An $SO(p+q,p)-$Higgs bundle in the regular fibres of the Hitchin fibration is semistable iff for all $\tilde{\sigma}_C$-invariant isotropic subbundles $M' \subset M$, we have that $\deg(M') \le 0$.
 \end{Proposition}
\begin{proof}
Note that a $\tilde{\sigma}_C$-invariant subbundle $M' \subseteq M$ induces a subbundle $V' \subseteq V_0$ by taking invariant direct image. Conversely every subbundle $V' \subseteq V_0$ is seen to arise in this manner. Examining behaviour of the invariant direct image around branch points, one sees that $deg(V') = deg(M') + s$, where $s$ is the number of ramification points $r \in C$ for which $M'_r \cap M_r^- \neq \{ 0 \}$. However, if $M'$ (or equivalently $V'$) is isotropic, then we must have $M'_r \cap M_r^- = \{ 0 \}$, because $M_r^-$ is a $1$-dimensional non-degenerate subspace of $M_r$, so its only isotropic subspace is $\{ 0 \}$. Therefore $deg(V') = deg(M')$ for isotropic subbundles.
\end{proof}

\begin{Remark}
It is interesting to compare the above stability conditions  for these orthogonal Higgs bundles with signature  with the conditions obtained in \cite{xia} for certain unitary Higgs bundles with signature, where Hodge bundles were used as extension bundles.  
\end{Remark}

\begin{Remark}
In the case of $q=1$, as one would expect,   the above conditions becomes obsolete, as can be see in \cite{ort}. 
\end{Remark}

We have now obtained all the results that lead to a geometric description of 
the intersection of the real slice $\mathcal{M}_{SO(p+q,p)}$ with the   fibres of the $SO(2p+q,\C)$ Hitchin fibration. 

\begin{Theorem}\label{teorema1}
There is a correspondence between semistable $SO(p+q,p)$-Higgs bundles $(V,W,\beta)$ in the regular fibres of the $SO(2p+q,\C)$-Hitchin fibration and triples $(L, M, \tau )$, where
\begin{itemize}
\item[(I)] $L \in {\rm Jac}(\bar S)[2]$ is an orthogonal line bundle on the $p$-fold cover  \[\bar S:=\{\xi^p+a_1\xi^{p-1}+\ldots+\xi a_{p-1}+a_{p}=0\}\subset {\rm Tot}(K^2),\]
where $a=\{a_i\}$ with $a_i\in H^0(\Sigma, K^{2i})$, and $\xi$ the tautological section of the pullback of $K^2$.
\item[(II)] $M$ is an equivariant rank $q$-orthogonal bundle  on the 2-fold cover $C=\{\zeta^2=a_{p}\}\subset {\rm Tot}(K^{p})$ of type $(q-1,1)$ over each ramification point, with a choice of orientation, and satisfying the following semistability condition:  all invariant isotropic subbundles $M'\subset M$ have degree $\le 0$. 
\item[(III)] For each zero $x$ of $a_p$, an isometry $\tau_x : M_r^- \to L_{r'}$ where $r,r'$ are the corresponding zeros of $\xi , \zeta$ lying over $x$.
\end{itemize}

Two such pairs $(L,M , \tau) , (L' , M' , \tau')$ lying in the same fibre of the Hitchin map correspond to isomorphic $SO(p+q,p)$-Higgs bundles if and only if there is an isomorphism $\psi : L \to L'$ of orthogonal line bundles and an isomorphism $\varphi : M \to M'$ of equivariant orthogonal bundles under which $\tau' = \pm \psi \circ \tau \circ \varphi^{-1}$.

\end{Theorem}

   \section{Characteristic classes}\label{section5}

We consider the characteristic classes of $SO(p+q,p)$-Higgs bundles in terms of the Cayley and Langlands type correspondences. The maximal compact of $SO(p+q,p)$ is $S(O(p+q)\times O(p))$, so an $SO(p+q, p)$-Higgs bundle $(V\oplus W, \Phi)$ carries three topological invariants: the Stiefel-Whitney classes $\omega_1(W)$, $\omega_2(W), \omega_2(V)$. By a $K$-theoretic approach following the methods of \cite{slices,classes,ort}, we give a description of these classes in terms of spectral data, leading to \autoref{classesT}.
    
\subsection{Stiefel-Whitney classes of $W$ via the spectral curve}
We will show here that the Stiefel-Whitney classes of $W$ can be described completely in terms of the associated line bundle $L \in Jac(\bar S)[2]$ of order $2$. Choosing a theta characteristic $K^{1/2}$, we may assign to a vector bundle $\mathcal{W}$ with $O(n , \mathbb{C})$-structure an analytic mod 2 index
\begin{eqnarray}\label{hola10}
\varphi_{\Sigma}(\mathcal{W}) = \dim H^0(\Sigma, \mathcal{W} \otimes K^{1/2}) ~({\rm mod} ~2).
\end{eqnarray}
It follows  from \cite[Theorem 1]{classes} that
\begin{eqnarray}
\omega_2(\mathcal{W})= \varphi_{\Sigma}(\mathcal{W})+\varphi_{\Sigma}(\det(\mathcal{W})) \label{algo}
\end{eqnarray}

Since we would like to understand the characteristic classes of $W$ in terms of the corresponding line bundle $L \in {\rm Jac}(\bar S)[2]$ in Item (I) of \autoref{teorema1}, we will make a few comments here of how these are related. 
Adopting the notation of \cite[Section 5]{classes}, for any line bundle $\CL$ on $\bar S$ of order $2$, we define
\begin{eqnarray}
\bar\pi_!(\CL)=\bar \pi_*(\CL \otimes (K_{\bar S} \otimes \bar\pi^*K^*)^{1/2}). \label{orthogonal}
\end{eqnarray}
By relative duality, $\bar\pi_!(\CL)$ inherits an orthogonal structure. In particular, as in \autoref{section2}, we have 
    \begin{eqnarray}
  W= \bar \pi_!(L).\label{Wfromdata}
  \end{eqnarray}
One is thus in the setting of \cite[Theorem 8]{ort}, which leads to the following:
\begin{Proposition} The characteristic classes associated to the rank $p$ vector bundle $W$ of an $SO(p+q,p)$-Higgs bundle $(V\oplus W,\Phi)$ with spectral data $L$ on $\bar S$ as in (I) of \autoref{teorema1} are given by
\begin{equation*}
\begin{aligned}
 \omega_1(W)&={\rm Nm}(L)  \in H^1(\Sigma, \Z_2),\\
\omega_2(W)&=\varphi_{  \bar S}(L)+\varphi_\Sigma({\rm Nm}(L)) \in \Z_2, 
\end{aligned}
\end{equation*}
where $\varphi_\Sigma$ and $\varphi_{\bar S}$ are the analytic mod 2 indices, and ${\rm Nm}  : Jac(\bar S )[2] \to Jac(\Sigma)[2]$ the Norm map. 
\end{Proposition}

\subsection{Stiefel-Whitney classes of $V$}  Since we are working with $SO(p+q,p)$-Higgs bundles, we have $\omega_1(V) = \omega_1(W)$. Hence, all that remains is to compute $\omega_2(V)$, which we will do by computing $\omega_2(V \oplus W)$ and using:
\[
\omega_2(V \oplus W) = \omega_2(V) + \omega_2(W) + \omega_1(V) \cup \omega_1(W) = \omega_2(V) + \omega_2(W) + \omega_1(W) \cup \omega_1(W).
\]
On a compact Riemann surface the mod $2$ intersection form is alternating, so we have $\omega^2_1(W) = 0$, and thus
\[
\omega_2(V) = \omega_2(V \oplus W) + \omega_2(W).
\]

Our strategy for computing $\omega_2(V \oplus W)$ will be to reduce the problem to the $q=1$ case, which is more manageable, essentially due to the abelian structure of the fibres in this case. Recall the spectral data consists of the orthogonal line bundle $L \in Jac(\bar S)[2]$, the quadratic bundle $(V_0 , Q_0)$ and for each zero $x$ of $a_p$, an isometry $\tau_x : N_x \otimes K_x^p \to L_{r}$, where $r$ is the zero of $\xi$ lying over $x$.\\

\noindent {\bf Case $q=1$}. Consider first the $SO(p+q,p)$-Higgs bundles for  $q=1$. Then since $V_0$ has rank $1$ and determinant $K^{-p}$, we have $V_0 \cong K^{-p}$. The quadratic form $Q_0$ can be viewed as multiplication by $a_p$ via $a_p : K^{-p} \otimes K^{-p} \to \mathbb{C}$. Therefore, if $x$ is a zero of $a_p$, we have $N_x \otimes K^p_x = K_x^{-p} \otimes K^p_x = \mathbb{C}$ equipped with $Q_0/a_p = a_p/a_p = 1$, the standard inner product. An isometry $\tau_x : \mathbb{C} \to L_{r}$ is equivalent to a choice of unit vector $\tau_{x} \in L_{r}$. Given such a collection $\{ \tau_x\}$, we seek to determine $\omega_2(V)$. First note that $V$ is an extension of the form
\[
0 \to K^{-p} \to V \to V_1 \cong W \otimes K \to 0
\]
and therefore, the $SO(2p+1,\mathbb{C})$-bundle $V \oplus W$ is an extension of the form
\[
0 \to K^{-p} \to \left( V \oplus W \right) \to F \otimes K^{1/2} \to 0,
\]
where $F$ is the $Sp(2p,\mathbb{C})$-bundle $F = (W \otimes K^{1/2} ) \oplus (W \otimes K^{-1/2})$. This is a special case of the construction of $SO(2p+1,\mathbb{C})$-Higgs bundles from $Sp(2p,\mathbb{C})$ considered in \cite{langlands}. It now follows from \cite{langlands} that $\omega_2(V \oplus W)$ can be obtained from $\{ \tau_x \}$ in the following manner. 

Recall that $SO(p+1,p)$-Higgs bundles define a  double cover $\rho : S \to \bar S$ and the involution $\sigma : S \to S$ from \eqref{hola} in \autoref{Higgs}. Let $Prym(S , \bar S)$ be the corresponding Prym variety, i.e. the subvariety of $Jac(S)$ given by line bundles $A \in Jac(S)$ satisfying $\sigma^* A \cong A^*$. Viewing the orthogonal structure on $L \in Jac(\bar S)[2]$ as a map $Q_{L}: L\to L ^*$, one can see that   $\rho^* L \in Prym(S , \bar S)$. Here,   the isomorphism $\alpha : \sigma^*(\rho^* L) \to (\rho^* L)^*$ is given by
\[
\sigma^* \rho^* L \cong \rho^* L \buildrel \rho^*(Q_{L}) \over \longrightarrow \rho^* L^* = (\rho^* L)^*,
\]
where the first isomorphism is given $\rho \circ \sigma = \rho$. Suppose that $N \in Prym(S , \bar S)$ satisfies $N^2 \cong \rho^* L$ and choose a specific isomorphism $\mu : N^2 \to \rho^* L$.  Then, there is an isomorphism $\nu : \sigma^* N \to N^*$ which we can assume is chosen so there is   a commutative diagram
\[
\xymatrix{
\sigma^* \rho^* L \ar[r]^-\alpha & \rho^* L^* \\
\sigma^* N^2 \ar[u]^-{\sigma^* \mu} \ar[r]^-{\nu^{\otimes 2}} & N^{-2} \ar[u]_-{\mu^{\otimes (-1)}}
}
\]
For a given $\mu$, this uniquely determines $\nu$ up to an overall sign. Moreover, $\nu$ can be viewed as a section of $N^* \sigma^*(N^*)$.

 Let $r \in \bar S$ be a zero of $\xi$ and $r'$ the unique point of $r' \in S$ lying over $r$. These are the ramification points of $S \to \bar S$ and they are also the fixed points of $\sigma : S \to S$. Thus $\nu^{-1}_{r'} \in N_{r'} N_{\sigma(r')} = N^2_{r'}$. Then if we set $\tau_x = \mu_{r'}( \nu^{-1}_{r'} ) \in (\rho^* L)_{r'} = L_r$, the above commutative diagram implies $Q_{L}( \tau_x , \tau_x ) =1$, so $\{ \tau_x \}$ is a choice of extension data. It follows from \cite[\textsection 4.3]{langlands} that $\omega_2(V \oplus W) = 0$ if and only if the extension data $\{ \tau_x \}$ arises in this way for some $N \in Prym(S , \bar S)$. This completely determines $\omega_2(V \oplus W)$ in terms of $N \in {\rm Prym}(S , \bar S)$ and the extension data $\{ \tau_x \}$.\\

\noindent {\bf The $q>1$ case}. 
The following proposition allows us to reduce the study of characteristic classes of $SO(p+q,p)$-Higgs bundles  to the $q = 1$ situation:
\begin{Proposition}\label{prop:split}
There exists a $\mathcal{C}^\infty$-isomorphism of vector bundles
\[
V_0 \cong V'_0 \oplus K^{-p},
\]
where $V'_0$ is a rank $q-1$ orthogonal vector bundle, such that $Q_0$ is the orthogonal direct sum of $V'_0$ with $K^{-p}$ equipped with the bilinear form $a_p : K^{-p} \otimes K^{-p} \to \mathbb{C}$.
\end{Proposition}
\begin{proof}
In the case $q = 1$ we have already seen that $V_0 \cong K^{-p}$, so let us assume $q \ge 2$. Choose disjoint open discs $D_x$ around each zero $x$ of $a_p$ such that on each $D_x$ there is a local holomorphic coordinate $z$ centred at $x$ with $a_p = z (dz)^{2p}$. We can further assume that over each $D_x$, there is a holomorphic frame $e_1 , \dots , e_q$ of $V_0$ with $Q_0(e_i , e_j) = \delta_{ij}$ for $(i,j) \neq (1,1)$ and $Q_0(e_1 , e_1) = z$. Set $e' = e_1 \otimes (dz)^p$. Then $e'$ is a section of $V_0 \otimes K^p$ defined on each $D_x$ and satisfying $Q_0( e' , e') = a_p$. Moreover, the space $\Sigma^* = \Sigma \setminus \cup_x D_x$ is homotopy equivalent to a wedge of circles, and   $(V_0 \otimes K^p , Q_0/a_p)$ is a rank $q$ orthogonal bundle on $\Sigma^*$. Choosing a reduction of structure of $V_0 \otimes K^p$ to the maximal compact $O(q) \subset O(q, \mathbb{C})$, since $q \ge 2$, it follows that the fibres of the unit sphere bundle of $V_0 \otimes K^p$ are connected. Hence, by obstruction theory we can find a smooth section $e$ of $V_0 \otimes K^p$ on $\Sigma^*$ with $Q_0(e,e) = a_p$. For each $x$, let $D'_x \subset D_x$ be a smaller open disc around $x$ so that $\bar D_x \setminus D'_x$ is an annulus. Then since $q \ge 2$ we can smoothly extend $e$ over the annulus so that $e |_{\partial D'_x} = e'$, and extend $e$ over $D'_x$ to equal $e'$. Let $E \subset V_0$ be given by $E = e K^{-p} \subset V_0$, and consider $V'_0$ be defined as the orthogonal complement of $E$ away from the zeros of $a_p$ and $V'_0|_{D'_x} = {\rm span}(e_2 , \dots , e_q)$. Then we have an orthogonal direct sum $V_0 = V'_0 \oplus E$. Note that $e : K^{-p} \to E$ gives an isomorphism $E \cong K^{-p}$ and since $Q_0(e , e) =a_p$, we can identify $E$ with $K^{-p}$ equipped with the bilinear form $a_p : K^{-p} \otimes K^{-p} \to \mathbb{C}$.
\end{proof}

Recall from \autoref{section4} that the vector bundle $V$ can be reconstructed as an extension of $V_0^*$ by $V_1^*$ and that the extension class has the form $\delta(i)$, for some $i \in H^0( D , Hom( N , V_1^*))$. 
  In general,   a decomposition $V_0 \cong V'_0 \oplus K^{-p}$ as in \autoref{prop:split} can only be done smoothly and not holomorphically. However, examining the proof of \autoref{prop:split}, we see that the isomorphism $V_0 \cong V'_0 \oplus K^{-p}$ can be chosen so that it is holomorphic in a neighbourhood of each zero of $a_p$. The choice of such an isomorphism, $\varphi : V_0 \to V'_0 \oplus K^{-p}$ in particular induces an identification $\varphi : N_x \to K_x^{-p}$, where $N_x = Ker(Q_0)_x \subseteq (V_0)_x$ as before. Hence, we can view $i$ as $\varphi(i) \in H^0( D , Hom( K^{-p} , V_1^* ) )$. Replacing $(V_0 , Q_0)$ by $(K^{-p} , a_p)$, we may construct from $i$ an extension of $V_1$ by $K^{-p}$:
\[
0 \to K^{-p} \to V'' \to V_1 \to 0.
\]

\begin{Theorem}\label{classesT}\label{teorema2}
Choose a smooth splitting $V_0 \cong V'_0 \oplus K^{-p}$ as in \autoref{prop:split} and use this to identify $N_x$ with $K^{-p}_x$. The extension data $\tau_x : N_x \otimes K^p_x \to L_r$ for $V$ is then identified with a collection $\{ \tau_x \}$ of unit vectors $\tau_x \in L_r$. Let $\delta = 0$ if $\{ \tau_x \}$ is induced from some $N \in Prym(S , \bar S)$ with $N^2 = \rho^* L,$ and $\delta = 1$ otherwise. Then
\begin{equation*}
\begin{aligned} 
\omega_1(W)&={\rm Nm}(L)  \in H^1(\Sigma, \Z_2), \\
\omega_2(W)&= \varphi_{  \bar S}(L)+\varphi_\Sigma({\rm Nm}(L)) \in \Z_2, \\
\omega_2(V) &= \varphi_{  \bar S}(L)+\varphi_\Sigma({\rm Nm}(L)) + \omega_2(V'_0) + \delta  \in \Z_2.
\end{aligned}
\end{equation*}
where $\varphi_\Sigma$ and $\varphi_{\bar S}$ are the analytic mod 2 indices, and ${\rm Nm}  : Jac(\bar S )[2] \to Jac(\Sigma)[2]$ the Norm map.
\end{Theorem}
\begin{proof}
From the constructions of $V$ and $V''$ via $i$ in \autoref{section4}, one sees that we have a $\mathcal{C}^\infty$-isomorphism of orthogonal vector bundles
\[
V \cong V'_0 \oplus V''.
\]
Note that $\det(V'') = K^{-p} \det(V_1) = K^{-p} \otimes K^p \otimes \det(W) = \det(W)$ (at the level of holomorphic bundles), so $\omega_1(V'' \oplus W) = 0$. Therefore
\begin{equation*}
\begin{aligned}
\omega_2(V \oplus W) &= \omega_2( V'_0 \oplus (V'' \oplus W) ) \\
&= \omega_2(V'_0) + \omega_2( V'' \oplus W) + \omega_1(V'_0) \cup \omega_1(V'' \oplus W) \\
&= \omega_2(V'_0) + \omega_2(V'' \oplus W).
\end{aligned}
\end{equation*}
But note that $V'' \oplus W$ is an extension
$
0 \to K^{-p} \to (V'' \oplus W) \to (V_1 \oplus W) = F \otimes K^{1/2} \to 0,
$
where $F$ is the $Sp(2p,\mathbb{C})$-bundle $F = (W \otimes K^{1/2} ) \oplus (W \otimes K^{-1/2})$, and therefore  we have reduced the computation to the $q=1$ case. Thus $\omega_2(V'' \oplus W)$ is computed from the extension data $\{ \tau_x \}$ as described above. Putting it all together, the theorem follows. 
\end{proof}

\section{Abelianization for split and quasi-split real forms}\label{section6}\label{quasi-split}

We have seen that the spectral data describing regular fibres of the moduli space of $SO(p+q,p)$-Higgs bundles consists of an abelian part given by the Cayley type correspondence and a non-abelian part, given by the Langlands type correspondence. In what follows, we will consider the special cases of the split real forms $SO(p+1,p)$ and the quasi-split real forms $SO(p+2,p)$. In these cases we will see that non-abelian data can be abelianized, providing a novel description of the intersection of the moduli spaces of $SO(p+q,p)$-Higgs bundles with the Hitchin fibration. 

We first consider the case of the quasi split real forms $SO(p+2,p)$, which are not split in \autoref{q2} and then consider the split real forms $SO(p+1,p)$ in \autoref{q1}. Finally in   \autoref{extra-comp} we describe from a geometric perspective the extra components appearing in the $SO(p+1,p)$ moduli space, which were identified via Morse theory computations by Collier in his PhD thesis \cite{brian2}. These components can be seen from the spectral data, and we will comment on extra components that should appear for any $q>1$. 

   \subsection{Quasi-split real forms which are not split}\label{q2}
   Here we consider the case $q=2$, that is, Higgs bundles for the quasi-split real form $SO(p+2,p)$. From the study of Cameral covers, it is known that abelian data should exist describing Higgs bundles for quasi-split real forms \cite{peon}. Here we will use our spectral data constructions to provide a concrete description of the abelian structure of the fibres, completing the explicit description of abelian data for all Higgs bundles coming from quasi split real forms (the case of split real forms was described in \cite{thesis}, and the cases of $U(p,p)$ and $U(p+1,p)$ in \cite{umm} and \cite{peon2} respectively).

Under the assumptions of \autoref{teorema1}, $SO(p+2,p)$-Higgs bundles can be parametrised in the fibres of the $SO(2p+2,\C)$ Hitchin fibration as triples $(L, M , \tau )$ where $L\in Jac(\bar S)[2]$ is an orthogonal line bundle on $\bar S$, $M$ is an equivariant $SO(2,\mathbb{C})$ on the $2$-fold cover $C$ of type $(1,1)$ over each fixed point of $\sigma_C$ such that all invariant isotropic subbundles have degree $\le 0$ and $\tau = \{ \tau_x \}$ is the extension data. 

\begin{Theorem}\label{teorema4}
The intersection of the moduli space of $SO(p+2,p)$-Higgs bundles with a fibre of the $SO(2p+2,\C)$ Hitchin fibration over a point  defining the spectral curve  
\[
S:= \{  \eta^{2p}+a_1\eta^{2p-2}+\ldots + a_{p}=0\}
\]
is given by triples $(L, N , \tau )$, where $L\in Jac(\bar S)[2]$, $N \in Prym(C , \Sigma)$ and $\tau = \{ \tau_x \}$ is the extension data, considered modulo $\tau \sim  -\tau$. There is a natural abelian group structure on such triples, given by 
\[
(L, N , \tau) ( L' , N' , \tau') = (L\otimes L' , N \otimes N' , \tau \otimes \tau' ).
\]
As a group, this fibre is isomorphic to $Prym(C , \Sigma) \times (\mathbb{Z}_2)^{(4p^2+2p)(g-1)+1}$. In particular, the fibre is $2^{(4p^2+2p)(g-1)+1}$ copies of $Prym(C,\Sigma)$.
\end{Theorem}  
\begin{proof}
Since in the above parametrisation of $SO(p+2,p)$-Higgs bundles in \autoref{teorema1} as triples $(L, M , \tau )$, the bundle $M$ is an $SO(2,\mathbb{C})$-bundle, we have
\[
M = N \oplus N^*
\]
for some line bundle $N$ on $C$. Let $\tilde{\sigma}_C : M \to M$ denote the lift of $\sigma_C$ to $M$. Since $N,N^*$ are the only isotropic sub-bundles of $M$, we must have either $\tilde{\sigma}_C( N ) = N$ or $\tilde{\sigma}_C(N) = N^*$. However, the fact that $M$ is assumed to have type $(1,1)$ over each fixed point implies that $\sigma$ must exchange the isotropics, so $\tilde{\sigma}_C(N) = N^*$. In particular, we have $\sigma_C^*(N) \cong N^*$ so that $N$ belongs to the Prym variety $Prym(C,\Sigma)$ of the cover $C \to \Sigma$. 

Conversely for any $N \in Prym( C , \Sigma)$, we obtain a rank $2$ equivariant orthogonal bundle $M = N \oplus N^*$. Notice that since $\tilde{\sigma}_C$ exchanges $N$ and $N^*$, there are no invariant isotropic sub-bundles of $M$, so $M$ satisfies the stability condition. Note also that $N$ and $N^*$ are distinguished from one another by the choice of an orientation on $M$ (swapping $N$ and $N^*$ reverses the orientation on $M$). Fix an isomorphism $\varphi : \sigma^* N \to N^*$, and let $r \in C$ be a fixed point of $\sigma_C$. Then over $r$, the map $\varphi$ induces an isomorphism $\varphi_r : N_r \to N_r^*$, i.e. an orthogonal structure on $N_r$.

 The $-1$-eigenspace $M_r^- \subset N_r \oplus N_r^*$ is given by $\{ ( v , -\varphi_r(v) ) \; | \; v \in N_r \}$, and in this way we get an identification $M^-_r \cong N_r$ as orthogonal spaces. The extension data $\{ \tau_x \}$ can now be viewed as a collection of isometries $\tau_x : N_r \to L_{r'}$ ($r'$ is the corresponding point in $\bar S$). Notice also that the only orientation preserving isometries of $M = N \oplus N^*$ are given by $(a,b) \mapsto (ca , c^{-1}b)$ where $c \in \mathbb{C}^*$ is a constant. Thus two triples $(L, N , \tau)  , (L' , N' , \tau')$ define isomorphic Higgs bundles if and only if $L' \cong L$, $N' \cong N$ and $\tau' = \pm \tau$. To summarise, we have:

Only the last statement about the group structure of the fibres remains to be shown. First of all note that $Jac(\bar S)[2] \cong (\mathbb{Z}_2)^{2g_{\bar S}}$ and $g_{\bar S} = (2p^2-p)(g-1)+1$. Next, note that $a_p$ has $deg(K^{2p}) = 4p(g-1)$ zeros, hence the possible choices of $\tau$ for given $L,N$ forms the group $(\mathbb{Z}_2)^{4p(g-1)-1}$, where the $-1$ comes from identifying $\tau$ and $-\tau$. Thus the group of components of the fibre is isomorphic to $(\mathbb{Z}_2)^{2g_{\bar S} + 4p(g-1) - 1} = (\mathbb{Z}_2)^{(4p^2+2p)(g-1)+1}$. The group structure of the fibre is then an extension of $(\mathbb{Z}_2)^{(4p^2+2p)(g-1)+1}$ by $Prym(C , \Sigma)$. The extension must be split, since $Prym(C , \Sigma)$ is a divisible group.
\end{proof}

\subsection{Split real forms}\label{q1}
  
 The group $SO(p+1,p)$ is a split real form and as such the spectral data for the corresponding Higgs bundles can be described using the techniques developed in this paper, as well as by considering them as 2-torsion points in the complex Hitchin fibration \cite{ort}. In either case, it follows directly that we have abelian spectral data.  
 
 The techniques developed in the previous sections have allowed us to understand $SO(p+1,p)$-Higgs bundles $(V\oplus W, \Phi)$ whose Higgs filed has characteristic polynomial of the form  
 \begin{eqnarray}  \det(\Phi-{\rm Id}\eta) = \eta(\eta^{2p}+a_1\eta^{2p-2}+\ldots + a_{p}).\label{curvesplit} \end{eqnarray}
Let $D \subset \Sigma$ be the divisor of zeros of $a_p$,  and $D_{\bar S} \subset \bar S$ the divisor in $\bar S$ given by the zeros of the tautological section $\xi$ of $K^2$. For each point $x \in D$ there is a unique point $x' \in D_{\bar S}$ lying over $x$. In particular, there is a naturally defined bijection between the points of these divisors. 

Recall that $W = \pi_*(L)\otimes K^{p-1}$, where $L$ is an orthogonal line bundle on $S$. The tautological section $\xi=\eta^2 : L\rightarrow L\otimes K^2$ pushes down to give the map $\beta_F : W \rightarrow W \otimes K^2$ and the orthogonal structure on $L$ induces by relative duality an orthogonal structure $Q_W$ on $W$. Moreover, in this case the quadratic bundle $(V_0, Q_0)$ is simply $(K^{-p}, a_{p})$, where $a_p$ is viewed as a bilinear form $a_{p}:  K^{-p} \otimes K^{-p} \rightarrow \mathcal{O}$. Then, from \autoref{teorema1}, the remaining data required to construct an $SO(p+1,p)$-Higgs bundle is for each $x \in D$ a choice of unit vector $e_x \in L_{x'}$. 

\begin{Corollary} The spectral data for $SO(p + 1, p)$ consists of:
\begin{itemize}
\item An orthogonal line bundle $L$ on $\bar S$,
\item For each $x\in D$, a choice of unit vector $e_x \in L_{x'}$.
\end{itemize}
Two such pairs $( L , \{ e_x \} ) , ( L' , \{ e'_x \} )$ define isomorphic $SO(p+1,p)$-Higgs bundles if and only if $L \cong L'$ as orthogonal line bundles by some isomorphism $\varphi : L \to L'$ such that $\varphi( e_x ) = \pm e'_x$. 
\end{Corollary}

There is a natural abelian group structure on the fibres given by tensor product. Repeating the counting argument given in \autoref{teorema4},  the fibres are isomorphic to the group $\mathbb{Z}_2^{(4p^2+2p)(g-1)+1}$.

Given a line bundle $L \in Jac(\bar S)[2]$, in general, there is no preferred choice of unit vector in $L_{x'}$. Suppose that we vary the coefficients $(a_1, a_2,\ldots , a_{p})$ and hence also the spectral curves $S, \bar{S}$ in a continuous family. Suppose that we also continuously vary the orthogonal line bundle $L$. For some fixed member of the family, choose for each $x \in D$ a unit vector $e_x \in L_{x'}$. Moving around a non-contractible loop in the family, we may find that the choice of unit vectors $\{e_x\}_{x\in D}$ does not extend over the loop. In other words, we may find a non-trivial monodromy action on the set of choices of unit vectors.

The situation, however, is much simpler in the special case that $L = \mathcal{O}$ is the trivial line bundle. Then for each $x \in D$ a choice of unit vector in $L_{x'}$ is simply a choice of either $+1$ or $-1$. Thus for each $x \in D$, we have $e_x \in \{+1,-1\}$. In this case it is easy to understand the monodromy action on $\{e_x\}_{x\in D}$. Namely, if we vary $(a_1, a_2,\ldots , a_{p})$ in some continuous loop within the space of smooth spectral curves, then the zeros of $a_{p}$ are moved around by some permutation $\theta:D\rightarrow D$. The monodromy action on the choice of unit vectors $\{e_x\}$ is just the natural action induced by $\theta$. In particular, we see that monodromy preserves the total number of $+1$'s and total number of $-1$'s. On the other hand, it is easy to see that the full permutation group of $D$ can be realised by monodromy, and thus the number of $+1$'s is the only monodromy invariant.

Note that $|D| = deg(K^{2p}) = 4p(g - 1)$, so the number of $+1's$ is an integer between $0$ and $4p(g - 1)$. Let $b_+$ denote the number of $+1$'s and $b_-$ the number of $-1$'s, so $b_+ + b_- = 4p(g - 1)$. Note that replacing $e_x$ by $-e_x$ for every $x \in D$ produces an isomorphic Higgs bundle. This operation exchanges the roles of $b_+$ and $b_-$, so without loss of generality we may assume that $b_+ \geq b_-$. It follows that there exists an integer $b$ such that
\begin{eqnarray}
b_+ =2p(g-1)+b,~{\rm and}~ b_- =2p(g-1)-b, \end{eqnarray}
where $0 \leq b \leq 2p(g - 1)$.
Denote by   $D_+$ and $D_-$ the set of $x \in D$ with $e_x = 1$ and $e_x = -1$ respectively, so that $K^{2p} = \mathcal{O}(D_+) \otimes \CO(D_-).$ Then, there exists unique up to scale sections  $s_+ \in \CO(D_+)$ and $s_-\in \CO(D_-)$ which vanish on $D_+$ and $D_-$ respectively. Scaling these sections appropriately, we can assume that
\begin{equation}\label{equ:splusminus}
s_-s_+=\frac{a_{p}}{2},
\end{equation}
where the factor of 2 is introduced for later convenience. Setting $B = K^{-p}(D_+)$, we have that $B^* = K^{-p}(D_-)$ and  $deg(B) = b_+ - 2p(g - 1) = b$.

\subsection{The extra components for $SO(p+1,p)$ from spectral data}\label{extra-comp}

It is known that the moduli space of $SO(p+1,p)$-Higgs bundles has {\it extra} components not detected by characteristic classes \cite{aparicio,brian2}. The extra components are {\it Hitchin-like}, meaning they share many similarities with the Hitchin component. They have been discovered as a byproduct of the Morse theoretic approach to counting connected components by looking for minima of the Hitchin functional. We will show here that these extra components emerge naturally from the spectral data point of view. Therefore, spectral data provides a simple conceptual explanation for the existence of these components.

 We have seen in \autoref{q1} that choosing $L = \CO$ and $e_x = \pm1$ for all $x \in D$ produces components distinguished by an integer invariant $b$. To see that these components are indeed the {\it extra} components of the moduli space of $SO(p + 1, p)$-Higgs bundles, we will carry out the reconstruction of the Higgs bundles corresponding to this spectral data and see that they are indeed those in the components described in \cite{brian2}. 
 
 In order to state the theorem, we introduce the following holomorphic differentials $\{ h_u \}$. Let 
 \begin{eqnarray}p(\xi,x) = \xi^{p} + a_1(x) \xi^{p-1} + \cdots + a_p(x)\label{polp}\end{eqnarray}
  be the characteristic polynomial of an $SO(p+1,p)$-Higgs bundle. For a given $x$, let $\xi_1 , \dots , \xi_p$ be the zeros of $p(\xi , x)$. Define $\{ h_u \}$ to be the complete homogeneous symmetric polynomials of $\xi_1 , \dots , \xi_p$. Namely for any $u \ge 1$, set
\begin{equation}\label{equ:homogsymm}
h_u = \sum_{1 \le i_1 \le i_2 \le \cdots \le i_u \le p } \xi_{i_1} \xi_{i_2} \cdots \xi_{i_p},
\end{equation}
and define $h_0 = 1$. Then $h_u$ is a well-defined section of $H^0( \Sigma , K^u )$. We recall the following version of Newton's identities, valid for all $j \ge 1$:
\begin{equation}\label{equ:newton}
\sum_{u=0}^j h_u a_{j-u} = 0.
\end{equation}
 
 \begin{Theorem}\label{teorema5}
Let $(W , V , Q_W , Q_V , \beta , \gamma)$ be the $SO(p+1,p)$-Higgs bundle associated to $L = \CO$, and $e_x = \pm 1$ according to whether $x \in D_+$ or $x \in D_-$. Let $B = K^{-p}(D_+)$ and let $s_+ , s_-$ be as in Equation \eqref{equ:splusminus}. Then, up to isomorphism $(W , V , Q_W , Q_V , \beta , \gamma)$ are given by:
\begin{equation*}
\begin{aligned}
W &= K^{p-1} \oplus K^{p-3} \oplus \cdots \oplus K^{-(p-1)}, \\
V &= \left( K^{p-2} \oplus K^{p-4} \oplus \cdots \oplus K^{-(p-2)} \right) \oplus B \oplus B^* = W_0 \oplus B \oplus B^*, \\
Q_W( w_i , w'_j ) &= \begin{cases} 0, & \text{ if } i+j < p+1, \\ h_{i+j-(p+1)} w_i w'_j, & \text{ if } i + j \ge p+1, \end{cases} \\
Q_{V} &= Q_{W_0} \oplus \left( \begin{matrix} 0 & a_{p}/2 \\ a_{p}/2 & 0 \end{matrix} \right), 
\end{aligned}
\end{equation*}
\begin{equation*}
\begin{aligned}
Q_{W_0}( v_i , v'_j ) &= \begin{cases} 0, & \text{ if } i+j < p, \\ h_{i+j-p}v_i v'_j, & \text{ if } i + j \ge p, \end{cases} \\
\beta( w_1 , \dots , w_p ) &= (w_1 - w_pa_{p-1} , w_2 - w_p a_{p-2} , \dots , w_{p-1} - w_p a_1 , w_p s_+ , -w_p s_- ), \\
\gamma( v_1 , \dots , v_{p-1} , g , h ) &= (s_+ h - s_- g , v_1 , \dots , v_{p-1} ).
\end{aligned}
\end{equation*}
Here, $w_i, w'_j , v_i , v'_j$ are respectively sections of $K^{(p+1)-2i}$, $K^{(p+1)-2j}$, $K^{p-2i}$, $K^{p-2j}$, $g$ is a section of $B$, $h$ is a section of $B^*$ and we identify $w_j,v_j$ with the corresponding sections $(0 , \dots , 0 , w_j , 0 , \dots 0)$, $(0 , \dots , 0 , v_j , 0, \dots , 0)$ of $W$ and $W_0$, and similarly for $w'_j , v'_j$.
\end{Theorem}

\begin{proof}
Since $W=\bar \pi_*(\CO)\otimes K^{p-1}$, we have
 \begin{eqnarray} W=K^{p-1}\oplus K^{p-3}\oplus \cdots\oplus K^{-(p-3)}\oplus K^{-(p-1)},
 \end{eqnarray}
 and therefore 
  \begin{eqnarray} 
  V_1^* = W\otimes K^{-1}=K^{p-2}\oplus K^{p-4}\oplus \cdots\oplus K^{-(p-2)}\oplus K^{-p}.
 \end{eqnarray}
In order to identify $Q_W$ and $\beta_F$ we need to make this isomorphism explicit. Consider 
 \begin{eqnarray}
 w=(w_1,w_2,\ldots w_p)\in \mathcal{C}^\infty( \Sigma , K^{p-1}\oplus K^{p-3}\oplus \cdots\oplus K^{-(p-3)}\oplus K^{-(p-1)}).
 \end{eqnarray}
 Then we identify $w$ with the section of $W = \overline{\pi}_*( K^{p-1} )$ given by:
\begin{equation}\label{eq:pushfwd}
w_1 + \xi w_2 + \xi^2 w_3 + \cdots + \xi^{p-1} w_p.
\end{equation}
Recall that $\beta_F : W \to W \otimes K^2$ is obtained by pushing down multiplication by $\xi$. Thus, if $w$ is given as in \eqref{eq:pushfwd}, then
\begin{equation*}
\begin{aligned}
\beta_F w &= \xi w_1 + \xi^2 w_2 + \cdots + \xi^{p-1} w_{p-1} + \xi^{p} w_p \\ 
& = -w_p a_p + \xi (w_1 - w_p a_{p-1} ) + \xi^2 (w_2 - w_p a_{p-2} ) + \cdots + \xi^{p-1} (w_{p-1} - w_p a_1 ).
\end{aligned}
\end{equation*}
In other words, we have:
\[
\beta_F(w_1 , w_2 , \dots , w_p) = (0 , w_1 , w_2 , \dots , w_{p-1}) -w_p(a_p , a_{p-1} , \dots , a_1 ).
\]

In order to compute $Q_W$, recall that if $f(\xi) = (\xi - \xi_1)(\xi - \xi_2) \cdots (\xi - \xi_p)$ a monic degree $p$ polynomial with distinct roots, we have
\[
\sum_i \frac{ \xi_i^r }{ p'(\xi_i) } = \begin{cases} 0, & \text{ if } r < p-1, \\ h_{r-(p-1)} & \text{ if } r \ge p-1, \end{cases}
\]
where $\{ h_u \}$ are the complete homogeneous symmetric polynomials of $\xi_1 , \dots , \xi_p$, as in \eqref{equ:homogsymm}.
Then, since $Q_W$ is obtained by relative duality:
\[
Q_W( w , w' )(x) = \sum_{ \{ x' \; | \; \overline{\pi}(x') = x \} }  \frac{ w(x') w'(x') }{ d \overline{\pi}(x') }.
\] 
and $\bar\pi:\bar S\rightarrow S$ is  the zero set of $p(\xi,x)$ as in \eqref{polp},
 with  $d\overline{\pi} = \partial_\xi p(x,\xi)$, so that
\[
Q_W( w , w')(x) = \sum_{ \{ \xi_i \; | \; p(\xi_i , x) = 0 \} } \frac{ w(\xi_i) w'(\xi_i) }{ \partial_\xi p(x , \xi_i ) }.
\]

Let $w_i$ be a section of $K^{(p+1)-2i}$ and $w'_j$ a section of $K^{(p+1)-2j}$. As in the statement of the theorem, we identify $w_i,w'_j$ with the corresponding sections of $W$. Then
\begin{equation*}
\begin{aligned}
Q_W( w_i , w'_j ) = \sum_{ \{ \xi_i \; | \; p(\xi_i) = 0 \} } \frac{ w_i w'_j \xi_i^{i+j-2}}{ p'(\xi_i) }  
= \begin{cases} 0, & \text{ if } i+j < p+1, \\ h_{i+j-(p+1)} w_i w'_j, & \text{ if } i + j \ge p+1. \end{cases}
\end{aligned}
\end{equation*}
We would like to calculate the quadratic form $Q_W( \beta_F v ,v' )$ on $V_1^* \cong W \otimes K^{-1}$, and for this it is better to make a change of basis and so we consider the following bundle automorphism
\begin{equation}\label{equ:psi}
\psi : W \otimes K^{-1} \to W \otimes K^{-1}, \quad \psi(v_1 , \dots , v_p) = (v_1 , \dots , v_p) + v_p(a_{p-1} , \dots , a_1 , 0).
\end{equation}
Considering $\psi$ as an isomorphism $\psi : V_1^* \to W \otimes K^{-1}$, we denote by  $Q_1$   the quadratic form on $V_1^*$ obtained by pullback of $Q_w( \beta_F v , v')$ on $W \otimes K^{-1}$, that is
\[
Q_1( v , v') = Q_W( \beta_F \psi(v) , \psi(v') ).
\]
Let $v_i$ be a section of $K^{p-2i}$ and $v'_j$ a section of $K^{p-2j}$. If $i,j < p$, one finds   that
\[
Q_{1}( v_i , v'_j ) = \begin{cases} 0, & \text{ if } i+j < p, \\ h_{i+j-p}v_i v'_j, & \text{ if } i + j \ge p. \end{cases}
\]
If $i < p$ and $j=p$, we have from Equation \eqref{equ:newton} that 
\begin{equation*}
\begin{aligned}
Q_1( v_i , v'_p ) &= Q_W( \underset{(i+1)\text{-th slot}}{(0, \dots , 0 , v_i , 0 , \dots , 0)} , v'_p(a_p , \dots , a_1 , 1) ) 
= v_i v'_p \sum_{j=0}^{i} a_j h_{i-j} 
= 0,
\end{aligned}
\end{equation*}
 Lastly, if $i = j = p$, then one finds that $\beta_F( \psi(v_p)) = -v_p a_p( 1 , 0 , \dots, 0)$. Therefore
\begin{equation*}
\begin{aligned}
Q_1( v_p , v'_p ) = Q_W( -v_pa_p(1,0, \dots , 0) , v'_p(a_p , \dots , a_1 , 1) ) 
= -v_p v'_p a_p.
\end{aligned}
\end{equation*}

We write
\begin{equation*}
\begin{aligned}
V_1^* &= K^{p-2} \oplus K^{p-4} \oplus \cdots \oplus K^{-(p-2)} \oplus K^{-p} \\
&= \left( K^{p-2} \oplus K^{p-4} \oplus \cdots \oplus K^{-(p-2)} \right) \oplus K^{-p} \\
&= W_0 \oplus K^{-p},
\end{aligned}
\end{equation*}
where
\[
W_0 = K^{p-2} \oplus K^{p-4} \oplus \cdots \oplus K^{-(p-2)}.
\]
Having  shown that $W_0$ and $K^{-p}$ are orthogonal with respect to $Q_1$, and that
%\[
$Q_1( v_p , v'_p ) = - a_p v_p v'_p,$
%\]
denote the restriction of $Q_1( \beta_F \; , \; )$ to $W_0$ by $Q_{W_0}$. By the above calculations, this agrees with the definition of $Q_{W_0}$ given in the statement of the theorem. Next, recall that to construct $V$ from special data as in \autoref{teorema1}, we first take the bundle 
\[
V' = (W\otimes K^{-1}) \oplus V_0 = (W \otimes K^{-1} ) \oplus K^{-p} = W_0 \oplus \left( K^{-p} \oplus K^{-p} \right),
\]
which has a degenerate quadratic form $Q_{V'}$. This form    is the direct sum of $Q_W( \beta_F \; , \; )$ with $a_{p}$ on the second $K^{-p}$ factor, and thus
\[
Q_{V'} = Q_{W_0} \oplus \left( \begin{matrix} -a_{p} & 0 \\ 0 & a_{p} \end{matrix} \right).
\]
Next we make a change of basis $K^{-p} \oplus K^{-p} \to K^{-p} \oplus K^{-p}$, so that the $K^{-p}$ factor which comes from $W\otimes K^{-1}$ is sent to the anti-diagonal $\{ ( w , -w ) \}$ and the factor of $K^{-p}$ which comes from $V_0$ is sent to the diagonal $\{ ( w , w) \}$. In such a basis $Q_{V'}$ becomes
\[
Q_{V'} = Q_{W_0} \oplus \left( \begin{matrix} 0 & a_{p}/2 \\ a_{p}/2 & 0 \end{matrix} \right).
\]
In this basis, the isotropic subspaces of $K^{-p} \oplus K^{-p}$ are the two $K^{-p}$ summands. Recall from \autoref{prop:merogamma} that to get $V$, we define $\mathcal{O}(V)$ to be the sheaf of meromorphic sections of $V'$ which for each $x \in D$ are allowed to admit first order poles on one of the two isotropics $\Gamma_x = K^{-p}_x \oplus 0$, or $\Gamma_x = 0 \oplus K^{-p}_x$. The choice of one of these two isotropics corresponds to whether $e_x = 1$ or $-1$ (which is which is unimportant, since changing the sign of every $e_x$ gives an isomorphic Higgs bundle). Thus we can assume that $D_+$ is the subset corresponding to the first isotropic and $D_-$ to the second. Therefore
\[
V = W_0 \oplus K^{-p}(D_+) \oplus K^{-p}(D_-) = W_0 \oplus B \oplus B^*.
\]
The induced quadratic form on $V$ is the direct sum of $Q_{W_0}$ with the natural dual pairing between $B$ and $B^*$.

Lastly, we need to work out the maps $\beta : W \to V \otimes K$ and $\gamma : V \to W \otimes K$. We have
\[
V = W_0 \oplus B \oplus B^* = \left( K^{(p-2)} \oplus K^{(p-4)} \oplus \cdots \oplus K^{-(p-2)} \right) \oplus B \oplus B^*
\]
and
$
W = K^{p-1} \oplus K^{p-3} \oplus \cdots \oplus K^{-(p-1)}.
$
Recall that $\beta$ is defined away from the zeros of $a_p$ by the natural inclusion $W \to (W\otimes K^{-1}) \otimes K \cong V_1^* \otimes K \subset V \otimes K$ and that this extends holomorphically to a map $\beta : W \to V \otimes K$. Bearing in mind that we are using the isomorphism $\psi$ of \eqref{equ:psi} to identify $V_1^*$ with $W \otimes K^{-1}$, we find:
\[
\beta( w_1 , w_2 , \dots , w_p ) = (w_1 -w_pa_{p-1}, w_2-w_pa_{p-2} , \dots , w_{p-1}-w_p a_1 , w_p s_+ , -w_p s_- ).
\]
Recall also that $\gamma$ is defined away from the zeros of $a_p$ by 
\[
V = V_1^* \oplus V_0 \buildrel (\psi , id) \over \longrightarrow \left( W \otimes K^{-1} \right) \oplus V_0 \to W \otimes K^{-1} \buildrel \beta_F \over \longrightarrow W \otimes K,
\]
where the map $\left( W \otimes K^{-1} \right) \oplus V_0 \to W \otimes K^{-1}$ is projection to the first factor. This map also extends holomorphically over $D$, and thus we   find
$\gamma( v_1 , v_2 , \dots , v_{p-1} , g , h ) = (s_+ h - s_- g , v_1 , \dots , v_{p-1} ).$
 \end{proof}

\begin{Remark}\label{extra}
From the above proof, the natural candidates for the extra  components conjectured to exist by  Guichard and Wienhard \cite[Conjecture 5.6]{anna2}  are those containing Higgs bundles whose spectral data $(L , M , \tau )$ in \autoref{teorema1} has the form $( \mathcal{O} , \mathcal{O}^q , \tau )$. Alternatively, this is equivalent to taking $SO(p+q,p)$-Higgs bundles of the form $(W , V \oplus \mathcal{O}^{q-1})$, where the pair $(W,V)$ is one of the $SO(p+1,p)$-Higgs bundles constructed in this section.  
\end{Remark}

%%%%%%%%%%%%%%%%%%%%%%%%%%%%%%%%
%%%%%%%%%%%%%%%%%%%%%%%%%%%%%%%%
%%%%%%%%%%%%%%%%%%%%%%%%%%%%%%%%
%%%%%%%%%%%%%%%%%%%%%%%%%%%%%%%%
%%%%%%%%%%%%%%%%%%%%%%%%%%%%%%%%
%%%%%%%%%%%%%%%%%%%%%%%%%%%%%%%%
%%%%%%%%%%%%%%%%%%%%%%%%%%%%%%%%
%%%%%%%%%%%%%%%%%%%%%%%%%%%%%%%%
%%%%%%%%%%%%%%%%%%%%%%%%%%%%%%%%
%%%%%%%%%%%%%%%%%%%%%%%%%%%%%%%%
%%%%%%%%%%%%%%%%%%%%%%%%%%%%%%%%

\section{Concluding remarks}\label{section-structures}\label{p2}\label{section-hermitian}

 \subsection{Groups of Hermitian type:
 $SO(2+q,2)$-Higgs bundles}   
In the case $p=2$, we have the group $SO(2+q,2)$ which is of Hermitian type and therefore $SO(2+q,2)$-Higgs bundles $(V,W,\beta)$ carry a Toledo invariant satisfying a Milnor-Wood type inequality. Whilst understanding the invariant and inequalities through representation theoretic methods might not be too direct, in what follows we show that  they have a  very concrete interpretation through the methods developed in this paper. 
 By \autoref{teorema1},  regular $SO(2+q,2)$-Higgs bundles are in correspondence with  $(L, M, \tau )$:
\begin{itemize}
\item[(I)] $L \in {\rm Jac}(\bar S)[2]$ is an orthogonal line bundle on the $2$-fold cover  satisfying \autoref{ass1} given by  $\bar S=\{\xi^2+a_1\xi^{p-1}+a_{2}=0\}\subset  {\rm Tot}(K^{2})$,  where $a_i\in H^0(\Sigma, K^{2i})$.
\item[(II)] $M$ is an equivariant rank $q$-orthogonal bundle  on the 2-fold cover $C=\{\zeta^2=a_{2}\}\subset {\rm Tot}(K^{2})$ of type $(q-1,1)$ over each ramification point, with a choice of orientation, and  such that all invariant isotropic subbundles $M'\subset M$ have degree $\le 0$. 
\item[(III)] For each zero $x$ of $a_2$, an isometry $\tau_x : M_r^- \to L_{r'}$, where $r,r'$ are the   zeros of $\xi , \zeta$ lying over $x$.
\end{itemize}
From \autoref{section2}   the data in (I) corresponds to a maximal $Sp(4,\mathbb{R})$-Higgs bundle\footnote{Recall that the Toledo invariant for such Higgs bundles is defined as the degree ${\rm deg}(W \otimes K^{1/2})=2g-2$.},  given by
\begin{equation*}
F = (W \otimes K^{1/2}) \oplus (W \otimes K^{-1/2}), \quad \quad \Phi_F = \left( \begin{matrix} 0 & \beta_F \\ Id & 0 \end{matrix} \right), ~{\rm ~where~}~\beta_F=\gamma \circ \beta.
\end{equation*}

From Gothen's work \cite{gothen}, the moduli space of maximal $Sp(4,\mathbb{R})$-Higgs bundles, or equivalently, of maximal $Sp(4,\mathbb{R})$ surface group representations has $3 \cdot 2^{2g} + 2g - 4$ connected components. Note that $\omega_{1}(W)=0$ if and only if $W=\mathcal{L}\oplus \mathcal{L}^{*}$ for some  $\mathcal{L}\in {\rm Jac}(\Sigma)$, which we may assume satisfies $c:=\deg(\mathcal{L}) \ge 0$. Then, there are  three types of components:
\begin{itemize}
\item[(a)]  $2^{2g}$ Hitchin components $M_{\mathcal{L}}$ (where the degree of $\mathcal{L}$ is maximal, in which case $\mathcal{L}^2=K$), 
\item[(b)] $2g-2$ components $M_{0,c}$ (where $\omega_1(W) = 0$ and $ c= {\rm deg}(\mathcal{L})$ is non maximal), and 
\item[(c)] $2(2^{2g}-1)$ components $M_{\omega_{1},\omega_{2}}$ given by the possible values of $(\omega_1(W) , \omega_2(W))$ with $\omega_{1}\neq 0$.
\end{itemize}
Note that Higgs bundles for $SO_0(2+q,2)$, the identity component of $SO(2+q,2)$, correspond to the cases in which $\omega_1(W) = 0$, i.e. components of types (a) and (b). In such cases we have $W=\mathcal{L}\oplus \mathcal{L}^*$, and when ${\rm Tr}(\beta_F)=0$ (these are referred to as {\it conformal Higgs bundles} in \cite{brian}), the induced $K^2$-twisted Higgs bundle $(W , \beta_F)$ is then a $K^2$-twisted $SL(2,\R)$-Higgs bundle (as opposed to $GL(2,\mathbb{R})$).

Via the Cayley type correspondence, the above classification of maximal $Sp(4,\mathbb{R})$-Higgs bundles into classes (a)-(c) gives a similar categorization of $SO(2+q,2)$-Higgs bundles into classes (a)-(c). In order to complete this to a description of connected components, one would also need to understand additional invariants involved in the construction, arising from the quadratic bundle and the extension data. 

In the case of maximal Higgs bundles, i.e. those where $(W , \beta_F)$ is of type (a), it was shown in \cite{steve1} that for $q>2$ the moduli space of representations into $SO_0(2+q,2)$ with maximal Toledo invariant, i.e. $\omega_1(W) = 0$ and $\deg(\mathcal{L})=2g-2$, has $2^{2g}$ {\it Hitchin type} connected components. Using the description of spectral data in \autoref{teorema1}, one can see these components in terms of maximal $Sp(4,\R)$-Higgs bundles, where they correspond to the $2^{2g}$ Hitchin components. The $2g-2$ components in (b)  are those referred to as {\it Gothen components} in \cite{brian2}, and they are the orthogonal version of the  $Sp(4, \R)$-representations discovered by Gothen in \cite{gothen}. 
From Higgs bundles which do not reduce to the identity component of $SO(2+q,2)$, there are $2(2^{2g}-1)$   values of $( \omega_1(W) , \omega_2(W))$, and  from \autoref{teorema1},   the other invariants introduced in \autoref{teorema2} should label further components of the moduli space of  $SO(2+q,2)$-Higgs bundles.

\begin{Remark}Recall that the natural  $Sp(4,\R)$-Higgs bundle $(F,\Phi_F)$ associated to an $SO(2+q,2)$-Higgs bundle $(E,\Phi)$ in \autoref{section2} has reduced spectral curve given by  
$ \bar S=\{\xi^2+a_1 \xi +a_2=0\}$. Noting that $a_1=-{\rm Tr}(\beta_F)$, we see that when the Higgs bundle is conformal, $a_1 = 0$ and $\bar S$ is given by $\xi^2 + a_2 = 0$. This is the same as the equation defining the auxiliary spectral curve $C$, i.e. $\bar S = C$ for conformal $SO(2+q,2)$-Higgs bundles.
\end{Remark}

%%%%%%%%%%%%%%%%%%%%%%%%%%%%%%%%
%%%%%%%%%%%%%%%%%%%%%%%%%%%%%%%%
%%%%%%%%%%%%%%%%%%%%%%%%%%%%%%%%
%%%%%%%%%%%%%%%%%%%%%%%%%%%%%%%%
%%%%%%%%%%%%%%%%%%%%%%%%%%%%%%%%
%%%%%%%%%%%%%%%%%%%%%%%%%%%%%%%%
%%%%%%%%%%%%%%%%%%%%%%%%%%%%%%%%
%%%%%%%%%%%%%%%%%%%%%%%%%%%%%%%%
%%%%%%%%%%%%%%%%%%%%%%%%%%%%%%%%
%%%%%%%%%%%%%%%%%%%%%%%%%%%%%%%%
%%%%%%%%%%%%%%%%%%%%%%%%%%%%%%%%
%%%%%%%%%%%%%%%%%%%%%%%%%%%%%%%%
%%%%%%%%%%%%%%%%%%%%%%%%%%%%%%%%
%%%%%%%%%%%%%%%%%%%%%%%%%%%%%%%%
%%%%%%%%%%%%%%%%%%%%%%%%%%%%%%%%
%%%%%%%%%%%%%%%%%%%%%%%%%%%%%%%%
%%%%%%%%%%%%%%%%%%%%%%%%%%%%%%%%
%%%%%%%%%%%%%%%%%%%%%%%%%%%%%%%%
%%%%%%%%%%%%%%%%%%%%%%%%%%%%%%%%
%%%%%%%%%%%%%%%%%%%%%%%%%%%%%%%%
%%%%%%%%%%%%%%%%%%%%%%%%%%%%%%%%

\subsection{$Sp(2p+2q,2p)$-Higgs bundles}\label{sec:symplectic}

As mentioned in the introduction, most of our results for orthogonal Higgs bundles have corresponding counterparts for symplectic Higgs bundles. In this section we shall discuss these results, but we shall do so briefly since  their proofs are very similar to the orthogonal case.

From \autoref{def2}, one finds that an $Sp(2p+2q,2p)$-Higgs bundle is a triple $(V,W,\beta)$ given by:
\begin{enumerate}
\item{A rank $2p+2q$ symplectic bundle $(V,Q_V)$}
\item{A rank $2p$ symplectic bundle $(W,Q_W)$}
\item{A holomorphic bundle map $\beta : W \to V \otimes K$.}
\end{enumerate}
Let $\gamma : V \to W \otimes K$ be the symplectic transpose of $\beta$. One can recover the associated $Sp(4p+2q,\mathbb{C})$-Higgs bundle $(E,\Phi)$  by setting $E=V\oplus W$ with symplectic form
\[
( (x,y) , (x',y') ) = Q_V(x,x') - Q_W(y,y'),
\]
and Higgs field $\Phi:E\rightarrow E\otimes K$   given by
\begin{eqnarray}\Phi=\left(\begin{array}{cc}
              0&\beta\\
\gamma&0
             \end{array} \right).
\end{eqnarray} 
The characteristic polynomial of $\Phi$ is of the form
\begin{eqnarray}
{\rm det}(\eta-\Phi)= \eta^{2q}(\eta^{2p}+a_{1}\eta^{2p-2}+\ldots +a_{p-1}\eta^{2}+a_{p})^2.
\end{eqnarray}
We define the spectral curves $S , \bar S$ exactly as in the orthogonal case. We suppose that \autoref{ass1} holds, in particular that $S$ and $\bar S$ are smooth. Following \autoref{Higgs}, define $V_0 , V_1$ as in the orthogonal case, as well as maps $\gamma_+ , \beta_+$ and commutative diagrams as in \autoref{E+}. 

Letting $\beta_F = \gamma \beta : W \to W \otimes K^2$ as before, one finds that $\beta_F$ is symmetric with respect to $Q_W$ and hence in this case $(W , Q_W , \beta_F )$ is a $K^2$-twisted $GL(p , \mathbb{H})$-Higgs bundle. This is the Cayley data in the symplectic case. Under \autoref{ass1}, it can be shown that $(W , Q_W , \beta_F)$ corresponds to a principal $Sp(1,\mathbb{C})$-bundle on $\bar S$, in other words, a rank $2$ symplectic vector bundle $L \to \bar S$, which is the symplectic analogue of an orthogonal line bundle.

Next, consider $(V_0 , Q_0)$, where $Q_0 = Q_V|_{V_0}$. This is a skew-symmetric quadratic bundle, i.e. the skew-symmetric analogue of a quadratic bundle as defined previously. Over each zero $x$ of $a_p$, the null space of $Q_0$ is a $2$-dimensional symplectic subspace $N_x \subseteq (V_0)_x$. Considering again the auxiliary double cover $\pi_C : C \to \Sigma$, one finds that $V_0$ corresponds to an equivariant symplectic bundle $(M , Q_M , \tilde{\sigma}_C)$ on $C$ such that the $-1$-eigenspace $M_r^-$ of $\tilde{\sigma}_C$ over a ramification point $r \in C$ is a $2$-dimensional symplectic space. 

The extension data $\tau$ needed to reconstruct $V$ as an extension of $V_1$ by $V_0$ is easily seen to consist of symplectomorphisms $$\tau_x : M_r^- \to L_{r'}$$ of $2$-dimensional symplectic spaces. In particular, for each zero $x$ of $a_p$, the space of such isomorphisms is a torsor over $Sp(1,\mathbb{C}) \cong SL(2,\mathbb{C})$. This is the Langlands data of \autoref{teorema1} in the symplectic case.

\begin{Remark}\label{no-extra}An interesting point of contrast between the orthogonal and symplectic cases is that $O(1,\mathbb{C}) \cong \{ \pm 1\}$ is disconnected while $Sp(1,\mathbb{C}) \cong SL(2,\mathbb{C})$ is connected. In particular, this explains the absence of any ``extra" components in the moduli space of $Sp(2p+2q,2p)$-Higgs bundles.
\end{Remark}

 \begin{Remark}
The case of  $q=0$ is not a split real form, and   for these $Sp(2p,2p)$-Higgs bundles  the spectral data was described in \cite{nonabelian,thesis}. Here, the intersection of the moduli space with the regular fibres is given by   a $\mathbb{Z}_2$-quotient of a moduli space of semi-stable rank 2 parabolic bundles on $\bar S$, and it corresponds to the $K^2$-twisted $GL(p , \mathbb{H})$-Higgs bundle mentioned above. 
 \end{Remark}

\subsection{Langlands duality}

The appearance of Higgs bundles (and flat connections) within string theory and the geometric Langlands program has led researchers to study the {\it derived category of coherent sheaves} and the {\it Fukaya category} of these moduli spaces. Therefore, it has become fundamental to understand Lagrangian submanifolds of the moduli space of Higgs bundles supporting holomorphic sheaves ($A$-branes), and their dual objects ($B$-branes). 
\smallbreak

We conclude the paper with some comments on Langlands duality. This section will be conjectural, as it is currently not understood how the duality should work over singular fibres of the Hitchin fibration.
 \smallbreak
 
  Let $^L G_\C$ denote the Langlands dual group of $G_\C$. There is a natural identification of invariant polynomials for $G_\C$ and $^L G_\C$, giving an identification $\mathcal{A}_{G_\C} \simeq \mathcal{A}_{^L G_\C}$ of the Hitchin bases. The two moduli spaces $\mathcal{M}_{G_\C}$ and $\mathcal{M}_{^L G_\C}$ are then torus fibrations over a common base and their non-singular fibres are dual abelian varieties \cite{dopa}. Kapustin and Witten give a physical interpretation of this in terms of S-duality, using it as the basis for their approach to the geometric Langlands program \cite{Kap}. In this approach a crucial role is played by the various types of branes and their transformation under mirror symmetry. 
    Whilst it is understood that Langlands duality exchanges brane types,    the exact correspondence is not yet known.
        In the case of $(B,A,A)$-branes of $G$-Higgs bundles, we have  conjectured  the following: 
           \vspace{0.08 in}

     \noindent {\bf Conjecture} \cite{slices}. {\it 
     The support of the dual brane to $\MG$ is the moduli space $\mathcal{M}_{\check{H}}\subset \mathcal{M}_{^L \GC}$ of $\check{H}$-Higgs bundles where $\check{H}$ is the group associated to the Lie algebra $\check{\mathfrak{h}}$ in \cite[Table 1]{nad}.} 
  
   \vspace{0.05 in}
 \begin{Remark}One should note that, in contrast with the $(A,B,A)$ and $(A,A,B)$ branes considered in \cite{slices}, for any $q>1$ the  $(B,A,A)$-branes studied in this paper lie completely over the singular locus of the $SO(2p+q,\C)$-Hitchin fibration. 
 \end{Remark}
 From the above conjecture, together with the geometric description that we have obtained of how the $(B,A,A)$-brane of $\M_{SO(p+q,p)}$ Higgs bundles intersects generic fibres of the $SO(2p+q,\mathbb{C})$ Hitchin fibration, we conjecture the following:   \vspace{0.08 in}

\noindent {\bf Conjecture}. {\it The $(B,A,A)$-brane of $\M_{SO(p+q,p)}$   inside $\M_{SO(2p+q,\mathbb{C})}$ has a dual $(B,B,B)$-brane in the Langlands dual moduli space   whose support consists of similar spaces embedded  through different maps:
   \begin{itemize}
   \item For $q$ odd: the dual support of the $(B,B,B)$-brane is $\M_{Sp(2p,\C)}\subset \M_{Sp(2p+q-1,\C)},$
      \item For $q$ even: the dual support of the $(B,B,B)$-brane is $\M_{SO(2p+1,\C)}\subset \M_{SO(2p+q,\C)},$
    \end{itemize}
We further conjecture that the hyperholomorphic sheaf supported on these spaces, giving the $(B,B,B)$-brane, is also independent of $q$.}   \vspace{0.08 in}

It is interesting to note that the support of branes for $q$ odd and even are dual to each other as hyperk\"ahler moduli spaces of complex Higgs bundles.

%%%%%%%%%%%%%%%%%%%%%%%%%%%%%%%%
%%%%%%%%%%%%%%%%%%%%%%%%%%%%%%%%
%%%%%%%%%%%%%%%%%%%%%%%%%%%%%%%%
%%%%%%%%%%%%%%%%%%%%%%%%%%%%%%%%
%%%%%%%%%%%%%%%%%%%%%%%%%%%%%%%%

%%%%%%%%%%%%%%%%%%%%%%%%%%%%%%%%
%%%%%%%%%%%%%%%%%%%%%%%%%%%%%%%%
%%%%%%%%%%%%%%%%%%%%%%%%%%%%%%%%
%%%%%%%%%%%%%%%%%%%%%%%%%%%%%%%%
%%%%%%%%%%%%%%%%%%%%%%%%%%%%%%%%
%%%%%%%%%%%%%%%%%%%%%%%%%%%%%%%%
%%%%%%%%%%%%%%%%%%%%%%%%%%%%%%%%
%%%%%%%%%%%%%%%%%%%%%%%%%%%%%%%%
%%%%%%%%%%%%%%%%%%%%%%%%%%%%%%%%
%%%%%%%%%%%%%%%%%%%%%%%%%%%%%%%%
%%%%%%%%%%%%%%%%%%%%%%%%%%%%%%%%
%%%%%%%%%%%%%%%%%%%%%%%%%%%%%%%%
%%%%%%%%%%%%%%%%%%%%%%%%%%%%%%%%
%%%%%%%%%%%%%%%%%%%%%%%%%%%%%%%%
%%%%%%%%%%%%%%%%%%%%%%%%%%%%%%%%
%%%%%%%%%%%%%%%%%%%%%%%%%%%%%%%%
%%%%%%%%%%%%%%%%%%%%%%%%%%%%%%%%
%%%%%%%%%%%%%%%%%%%%%%%%%%%%%%%%
%%%%%%%%%%%%%%%%%%%%%%%%%%%%%%%%
%%%%%%%%%%%%%%%%%%%%%%%%%%%%%%%%
%%%%%%%%%%%%%%%%%%%%%%%%%%%%%%%%
%%%%%%%%%%%%%%%%%%%%%%%%%%%%%%%%
%%%%%%%%%%%%%%%%%%%%%%%%%%%%%%%%
%%%%%%%%%%%%%%%%%%%%%%%%%%%%%%%%
%%%%%%%%%%%%%%%%%%%%%%%%%%%%%%%%
%%%%%%%%%%%%%%%%%%%%%%%%%%%%%%%%
%%%%%%%%%%%%%%%%%%%%%%%%%%%%%%%%
%%%%%%%%%%%%%%%%%%%%%%%%%%%%%%%%

% \oddsidemargin=-0.35in
  \textheight 10.1in
 \newpage
\pagestyle{plain}

\end{document}